\documentclass[12pt,tbtags,leqno]{amsart}

\usepackage{amssymb}
\usepackage{amsthm}
\usepackage{amsmath}
\usepackage{verbatim}
\usepackage{color}

\numberwithin{equation}{section}

\newcommand{\HH}{\mathcal{H}}
\newcommand{\HK}{\mathcal{K}}
\newcommand{\HM}{\mathcal{M}}
\newcommand{\HE}{\mathcal{E}}

\newcommand{\HB}{\mathcal{B}}
\newcommand{\HN}{\mathcal{N}}

\newcommand{\WPH}{{{\mathcal H} \odot {\mathcal H}}}

\newcommand{\dB}{\partial \mathbb B_d}

\newcommand{\D}{\mathbb{D}}

\newcommand{\C}{\mathbb{C}}
\newcommand{\N}{\mathbb{N}}
\newcommand{\R}{\mathbb{R}}

\newcommand{\Hbar}{\overline{\mathcal{H}}}

\newcommand{\ran}{\mathrm{ran \ }}

\newcommand{\la}{\langle}
\newcommand{\ra}{\rangle}
\newcommand{\Bd}{\mathbb B_d}

\newcommand{\Hol}{\operatorname{Hol}}
\newcommand{\clos}{\operatorname{clos}}

\def\HE{{\mathcal E}}

\def\om{\omega}

\newcommand{\Ker}[1]{\mathsf{Ker}~}

\theoremstyle{plain}
\newtheorem{theorem}{Theorem}[section]
\newtheorem{lemma}[theorem]{Lemma}

\newtheorem{corollary}[theorem]{Corollary}

\newtheorem{definition}[theorem]{Definition}

\theoremstyle{definition}

\newcommand{\Mult}{\operatorname{Mult}}
\newcommand{\Han}{\operatorname{Han}}
\begin{document}

\bibliographystyle{plain}

\thanks{}

\date{\today}

\title[Weak products of complete Pick spaces]{Weak products of complete Pick spaces.}
\author[A. Aleman]{Alexandru Aleman}
\address{Lund University, Mathematics, Faculty of Science, P.O. Box 118, S-221 00 Lund, Sweden}
\email{alexandru.aleman@math.lu.se}

\author[M. Hartz]{Michael Hartz}
\address{Department of Mathematics, Washington University in St. Louis, One Brookings Drive,
St. Louis, MO 63130, USA}
\email{mphartz@wustl.edu}
\thanks{M.H. was partially supported by an Ontario Trillium Scholarship and a Feodor Lynen Fellowship}

\author[J. M\raise.5ex\hbox{c}Carthy]{John E. M\raise.5ex\hbox{c}Carthy}
\address{Department of Mathematics, Washington University in St. Louis, One Brookings Drive,
St. Louis, MO 63130, USA}
\email{mccarthy@wustl.edu}
\thanks{J.M. was partially supported by National Science Foundation Grant DMS 1565243}

\author[S. Richter]{Stefan Richter}
\address{Department of Mathematics, University of Tennessee, 1403 Circle Drive, Knoxville, TN 37996-1320, USA}
\email{srichter@utk.edu}


\begin{abstract} Let $\HH$ be the Drury-Arveson or Dirichlet space of the unit ball of $\C^d$. The weak product $\HH\odot\HH$ of $\HH$ is the collection of all functions $h$ that can be written as $h=\sum_{n=1}^\infty f_ng_n$, where $\sum_{n=1}^\infty \|f_n\|\|g_n\|<\infty$. We show that $\HH\odot \HH$ is contained in the Smirnov class of $\HH$, i.e. every function in $\HH\odot \HH$ is a  quotient of two multipliers of $\HH$, where the function in the denominator can be chosen to be cyclic in $\HH$. As a consequence we show that the map $\HN \to \clos_\WPH \HN$ establishes a 1-1 and onto correspondence between the multiplier invariant subspaces of $\HH$ and of $\WPH$.

  The results hold for many weighted Besov spaces $\HH$ in the unit ball of $\C^d$ provided the reproducing kernel has the complete Pick property.  One of our main technical lemmas states that for weighted Besov spaces $\HH$ that satisfy what we call the multiplier inclusion condition any bounded column multiplication operator $\HH \to \oplus_{n=1}^\infty \HH$ induces a bounded row multiplication operator  $\oplus_{n=1}^\infty \HH \to \HH$. For the Drury-Arveson space $H^2_d$ this leads to an alternate proof of the characterization of interpolating sequences in terms of weak separation and Carleson measure conditions.
\end{abstract}

\maketitle

%\tableofcontents

\section{Introduction}
By a Hilbert (Banach) function space on a set $X$ we mean a Hilbert (Banach) space $\HB$ which consists of complex-valued functions on $X$ such that for each point  $w\in X$ the point evaluation functional $f\to f(w)$  is bounded on $\HB$.  Associated with every Banach function space $\HB$ we have the collection of multipliers $\Mult(\HB)=\{\varphi: X\to \C: \varphi \HB \subseteq \HB\}$. This is another Banach function space when equipped with the norm $\|\cdot\|_M$ that equals the operator norm of the induced multiplication operator $M_\varphi: \HB \to \HB, f\to\varphi f$, $f\in \HB, \varphi\in \Mult(\HB)$. We will call a function $f$ cyclic in $\HB$, if the set $\{\varphi f: \varphi \in \Mult(\HB)\}$ is dense in $\HB$.

 Each Hilbert function space $\HH$ has a reproducing kernel, i.e. a function $k:X \times X \to \C$ such that $f(w)=\la f, k_w\ra $ for all $w\in X$. Here $k_w(z)=k(z,w)$.  We say a reproducing kernel is normalized, if there is a $z_0\in X$ such that $k_{z_0}=1$.
 A Banach (Hilbert) space of analytic functions will be a Banach (Hilbert) function space which is contained in $\Hol(\Omega)$, the collection of holomorphic functions on the open set $\Omega \subseteq  \C^d$ for some $d\in \N$.

 In this paper a normalized complete Pick kernel will be a normalized reproducing kernel of the type $k_w(z)=\frac{1}{1-u_w(z)}$, where $u_w(z)$ is positive definite, i.e. for all $n\in \N, z_1,...,z_n \in X, $ and $a_1,..., a_n \in \C$ we have $\sum_{i,j}a_i \overline{a}_j u_{z_j}(z_i) \ge 0$.
 An important example of such a complete Pick kernel is the Szeg{\H{o}} kernel $k_w(z)=(1-\overline{w}z)^{-1}$. It is the reproducing kernel for the Hardy space $H^2$ of the unit disc $\D$, and it is fair to say that the function and operator theories associated with Hilbert function spaces with complete Pick kernels share many properties with the corresponding theories of $H^2$. We refer the reader to \cite{AgMcC} and \cite{AlHaMcCRiM} for some examples of this.
  In particular, it is a useful fact  that  Hilbert function spaces $\HH$ with a normalized complete Pick kernel
   are contained in the Smirnov class $N^+(\HH)$ associated with $\HH$,  \cite{AlHaMcCRiM}, where
   $$N^+(\HH)=\Big \{f=\frac{\varphi}{\psi}: \varphi, \psi \in \Mult(\HH), \psi \text{ cyclic in }\HH \Big\}.$$

At this point we mention two further important  examples of spaces with complete Pick kernels where the previous remark applies. The Dirichlet space of the unit disc,
$D=\{f\in \Hol(\D): \int_\D|f'|^2 dxdy <\infty\},$ it has reproducing kernel $k_w(z)=\frac{1}{\overline{w}z}\log \frac{1}{1-\overline{w}z}$ (see e.g. \cite{AgMcC}, Corollary 7.41, for the verification that $k_w(z)$ is a CNP kernel), and the Drury-Arveson space $H^2_d$ of analytic functions in the unit ball of $\C^d$. It is defined by the reproducing kernel $k_w(z)=\frac{1}{1-\la z, w\ra}$, $\la z,w\ra=\sum_{i=1}^d z_i\overline{w}_i$. See   \cite{DiriPrimer} and \cite{Shalit} for  indepth information about these spaces.

 Let
  $\HH$ be a Hilbert function space on a set $X$ . The weak product of $\HH$ is defined by
$$\HH\odot\HH= \Big\{\sum_{i=1}^\infty f_ig_i: f_i,g_i \in \HH, \sum_{i=1}^\infty \|f_i\|\|g_i\|<\infty\Big\},$$
where the norm on $\HH\odot\HH$ is given by $$\|h\|_{\HH\odot\HH}=\inf\Big\{\sum_{i=1}^\infty \|f_i\|\|g_i\|: h= \sum_{i=1}^\infty f_ig_i\Big\}.$$
One verifies that $\HH\odot\HH$ is a Banach function space, \cite{RiSuWeakProd} for the case of spaces of analytic functions, but the general case can be proved the same way (also see Section \ref{BackgroundWPH}). It is known that $H^2(\partial \Bd)\odot H^2(\partial \Bd)=H^1(\partial \Bd)$ and $L^2_a(\Bd)\odot L^2_a(\Bd)=L^1_a(\Bd)$, and there are similar results for weighted Bergman spaces, \cite{CRW}. We think of $\HH \odot \HH$ an analogue of $H^1$ for the function theory of the Hilbert function space $\HH$. It is known that for many examples of spaces $\HH$ one observes an analogue of the $H^1$-$BMO$-duality and a connection to  Carleson measures and  the theory of Hankel operators on $\HH$. In this paper we will add to this circle of ideas for Hilbert function spaces with normalized complete Pick kernels. We refer the reader to \cite{CRW}, \cite{ARSWBilinearForms}, and \cite{RiSuWeakProd} for further motivation and details about weak products.

Note that even for the cases where $\HH$ equals the Dirichlet or the Drury-Arveson space it is  unclear whether there is a simple description of the functions in $\WPH$.

 \begin{theorem} \label{IntroSmirnov1a} Let $\HH$ be a separable Hilbert function space on the non-empty set $X$ such that  the reproducing kernel for $\HH$ is a normalized complete Pick kernel. Then $$\HH\odot \HH \subseteq  N^+(\WPH),$$
   where $$N^+(\WPH)=\Big\{\frac{\varphi}{\psi}: \varphi\in \Mult(\WPH), \psi \in \Mult(\HH), \psi \text{ cyclic in }\HH \Big\}.$$
\end{theorem}

Beurling's theorem implies that the nonzero multiplier invariant subspaces of $H^2$ and $H^1$ are given by $\varphi H^2$ and $\varphi H^1$ for some inner function $\varphi$. It turns out that for the spaces $\HH$ under consideration a similarly close relationship exists between the multiplier invariant subspaces of $\HH$ and $\HH\odot \HH$. In \cite {LuoRi} and \cite{RiSunkes} it was shown that if $\HH=D$ or $\HH=H^2_d$, then for every multiplier invariant subspace $\HM$ of $\HH$ we have $\HM=\HH\cap\clos_{\HH\odot \HH} \HM$. In this paper we will refine this type of connection and we will see that it holds for a much wider class of complete Pick spaces.

One easily checks  that for any Hilbert function space the contractive inclusion $\Mult(\HH)\subseteq  \Mult(\HH\odot\HH)$ holds. For certain first order weighted Besov spaces (including the Dirichlet space of the unit disc and the Drury-Arveson space $H^2_d$ for $d \le 3$) it was shown in \cite{RiWi} that $\Mult(\HH)= \Mult(\HH\odot\HH)$, but we do not know whether  such an equality holds in a more general setting.

 \begin{theorem} \label{IntroSmirnov1b} Let $\HH$ be a separable Hilbert function space on the non-empty set $X$ such that  the reproducing kernel for $\HH$ is a normalized complete Pick kernel.

 Then
 $\HM\cap \Mult(\WPH)$ is dense in $\HM$ for every $\Mult(\WPH)$-invariant subspace $\HM$ of $\WPH$.
 \end{theorem}
 In particular, this implies that every non-zero multiplier invariant subspace of $\WPH$ contains a non-zero multiplier.
 It follows that if $\Mult(\HH \odot \HH)$ contains no zero-divisors (for instance if $\Mult(\HH \odot \HH)$ consists of analytic functions on a domain in $\mathbb C^d)$, then the lattice of multiplier invariant subspaces of $\WPH$ is cellularly indecomposable, i.e. whenever $\HM, \HN$ are such invariant subspaces with $\HM \ne (0), \HN \ne (0)$, then $\HM\cap \HN \ne (0)$.
 For $\HH=D$ (the Dirichlet space of the unit disc) this together with Proposition 3.6 of \cite{RichterDiss} provides a new proof of the theorem of Luo's that says that all nonzero multiplier invariant subspaces $\HM$ of $D\odot D$ have index 1, i.e. they satisfy dim $\HM \ominus z\HM=1$, \cite{Luo}. Our Theorem also shows that the same results hold in other weighted Dirichlet spaces of the unit disc.

 \

Under a technical hypothesis that is satisfied for many weighted Besov spaces in the unit ball of $\C^d$ we obtain a strengthened version of the previous theorems.
Every sequence $\Phi=\{\varphi_1, \varphi_2,...\}\subseteq  \Mult(\HH)$ of multipliers of a Hilbert function space $\HH$ can be used to define a column operator  $\Phi^C: h\to (\varphi_1h, \varphi_2 h,...)^T$ and a row operator $\Phi^R:(h_1,h_2,...)^T \to \sum_{i\ge 1}\varphi_ih_i$. Here we have used $(h_1,...)^T$ to denote a transpose of a row vector. We write $M^C(\HH)$  for the set of bounded column multiplication operators $\HH \to \bigoplus_{n=1}^\infty \HH$ and $M^R(\HH)$ for the set of bounded row multiplication operators $\bigoplus_{n=1}^\infty \HH \to \HH$.
The useful technical hypothesis is the condition $M^C(\HH) \subseteq M^R(\HH)$, that is, every bounded
column multiplication operator on $\HH$ is also a bounded row multiplication operator. By the closed graph
theorem, such an inclusion is automatically continuous.

\begin{theorem} \label{IntroSmirnov2}
Let $\HH$ be a separable Hilbert function space on the non-empty set $X$ such that  the reproducing kernel for $\HH$ is a normalized complete Pick kernel.

If  $M^C(\HH)\subseteq  M^R(\HH)$ continuously, then
$$\HH\odot \HH \subseteq  N^+(\HH).$$
Furthermore,
 every $\Mult(\HH)$-invariant subspace of $\HH\odot \HH$ is $\Mult(\HH\odot \HH)$-invariant, and the map \begin{align*}&\eta:\HN \to \clos_{\WPH} \HN\end{align*}  establishes a  1-1 and onto correspondence between the multiplier invariant subspaces of $\HH$ and of $\WPH$. We have
\begin{enumerate}
\item $\HM=\clos_{\WPH}(\HM\cap \Mult(\HH))$ for every multiplier invariant subspace $\HM$ of $\WPH$, and
\item $\HN=\HH \cap \clos_{\WPH} \HN$ for every multiplier invariant subspace $\HN$ of $\HH$.
\end{enumerate}
\end{theorem}

 Trent showed that for the Dirichlet space $D$ of the unit disc $\D\subseteq  \C$ one has the continuous inclusion $M^C(D)\subseteq  M^R(D)$, but that $M^R(D) \nsubseteq M^C(D)$. In fact, he showed that for the Dirichlet space the norm of the inclusion is at most $\sqrt{18}$, see Lemma 1 of \cite{TrentCorona}. We will establish a generalization of Trent's Theorem to many weighted Besov spaces in the unit ball $\Bd$ of $\C^d$.

 A non-negative integrable function $\om$ on $\Bd$ is called an admissible weight, if  the weighted Bergman space $L^2_a(\om)=L^2(\om dV)\cap \Hol(\Bd)$ is closed in $L^2(\om dV)$, and if point evaluations $f\to f(z)$ are bounded on $L^2_a(\om)$ for each $z\in \Bd$.  Here $V$ is used to denote Lebesgue measure on $\C^d$ restricted to $\Bd$, normalized so that $V(\Bd)=1$. Radial weights are non-negative integrable functions such that for each $0\le r<1$  the value $\om(rz)$ is independent of $z\in \dB$, and one easily checks that a radial weight is admissible, if and only if $\int_{|z|>t} \om dV>0$ for each $t\in [0,1)$.   If $\om$ is radial, then we have $$\|f\|^2_{L^2_a(\om)}=\int_{\Bd}|f|^2 \om dV= \sum_{n\ge 0} \|f_n\|^2_{L^2_a(\om)},$$ where $f=\sum_{n\ge 0} f_n$ is the decomposition of the analytic function $f$ into a sum of homogeneous polynomials $f_n$ of degree $n$.

Let $R=\sum_{i=1}^d z_i\frac{\partial}{\partial z_i}$ denote the radial derivative operator, then $Rf= \sum_{n\ge 1} nf_n$. More generally, for each nonzero $t\in \R$ we may consider the "fractional" transformation $R^t:\sum_{n\ge 0}  f_n \to \sum_{n\ge 1} n^t f_n$.

 For a positive integer $N$ and an admissible weight $\om$ we define
\begin{align*}B^N_\om&=\{f\in \Hol(\Bd): R^N f \in L^2_a(\om)\},\\ \|f\|^2_{B^N_\om}&=\|\om\|_{L^1(V)}|f(0)|^2 + \int_{\Bd}|R^N f|^2\om dV.\end{align*}
 We also write $B^0_\om=L^2_a(\om)$.  We will say that $\HH$ is a weighted Besov space, if $\HH$ is a space of analytic functions on $\Bd$ such that there is an admissible weight $\om$ and a non-negative integer $N$ satisfying
  $\HH=B^N_\om$ with equivalence of norms.

 If $\om$ is an admissible radial weight, then  the spaces $B^N_\om$ are part of a one-parameter family of spaces defined for $s\in \R$ by
 \begin{align}\label{BesovDefinition} \|f\|^2_{B^s_\om}= \|f_0\|^2_{L^2_a(\om)}+ \sum_{n\ge 1} n^{2s}\|f_n\|^2_{L^2_a(\om)}<\infty,\end{align}
 where as above  $f=\sum_{n\ge 0} f_n\in \Hol(\Bd)$.

If $\om(z)=1$, $s\in \R$, and $f\in \Hol(\Bd)$,  then $f\in B^s_\om$ if and only if $R^sf\in L^2_a$, the unweighted Bergman space. Thus, in this case the collection $B^s_\om$ consists of standard weighted Bergman or Besov spaces. We have $B^{d/2}_{{\mathbf 1}}=H^2_d$, the Drury-Arveson space, $B^{1/2}_{{\mathbf 1}}=H^2(\dB)$, the Hardy space of the Ball, and for $s<1/2$ we obtain the weighted Bergman spaces $B^s_{\mathbf 1}= L^2_a((1-|z|^2)^{-2s}dV)$, where all equalities are understood to mean equality of spaces with equivalence of norms. These spaces have been extensively studied in the literature. We refer the reader to \cite{ZhaoZhu}, where the $L^p$-analogues of these spaces were considered as well. If $d=1$ and $s=1$, then $B^1_{\mathbf 1}=D$, the classical Dirichlet space of the unit disc. More generally, if $d=1$ and $s>1/2$, then these spaces are usually referred to as Dirichlet-type spaces, see \cite{BrownShields}.

If $\om_\alpha(z)=(1-|z|^2)^{\alpha}$ for some $\alpha >-1$, then $\om_\alpha$ is called a standard weight, and we obtain the same spaces as for $\om=1$, but with a shift in indices: $B^s_{\om_\alpha}= B^{s-\frac{\alpha}{2}}_{\mathbf 1}$. This can be verified by using polar coordinates in (\ref{BesovDefinition}) and the asymptotics $\int_0^1 t^n(1-t)^\alpha dt = \frac{\Gamma(n+1)\Gamma(\alpha+1)}{\Gamma(n+\alpha+2)}\approx n^{-\alpha-1}$, which follows e.g. from Stirling's formula. In particular, it follows that  $B^{s}_{\mathbf 1}$ is a weighted Besov space for any $s\in \R$.

\begin{definition}\label{IntroMultInclCond} We say that a weighted Besov space $\HH$ satisfies the multiplier inclusion condition, if there is an admissible weight $\om$ on $\Bd$ and a non-negative integer $N$ with $\HH=B^N_\om$ and  $$M^C(B^N_{\om})\subseteq  M^C(B^{N-1}_{\om})\subseteq  \cdots \subseteq  M^C(B^0_{\om})$$ with continuous inclusions.
\end{definition}

   \begin{theorem}\label{IntroRowColumnOp}  Let $\HH$ be a weighted Besov space that satisfies the multiplier inclusion condition.

   Then $M^C(\HH)\subseteq  M^R(\HH)$ and  there is a $c>0$ such that $$\|\Phi^R\|_{\HH}\le c\|\Phi^C\|_{\HH}$$ for all $\Phi\in   M^C(\HH)$.
  \end{theorem}

It is known and easy to verify that $M^C(L^2_a(\om))=M^R(L^2_a(\om))=H^\infty(\ell_2)$, where
$$H^\infty(\ell_2)=\Big\{(\varphi_1,\varphi_2,...): \varphi_j\in H^\infty \text{ and } \sup_{z\in \Bd} \sum_j |\varphi_j(z)|^2 < \infty\Big\}.$$ Similarly, it is a standard fact that for each $n\in \N$  one has $M^C(B^n_\om)\subseteq  H^\infty(\ell_2)$. Thus for $N=1$ every weighted Besov space satisfies the multiplier inclusion condition, and hence $M^C(B^1_\om)\subseteq  M^R(B^1_\om)$ holds for all admissible weights.

In Section \ref{SectionWeightedBesov} we will provide a short and elementary proof that shows that every Hilbert space of analytic functions in $\Bd$ whose reproducing kernel is of the type $k_w(z)= (1-\la z,w\ra)^{\alpha}$, $\alpha<0$, satisfies the multiplier inclusion condition. This includes the Drury-Arveson space.

  A second approach to the multiplier inclusion condition is via complex interpolation. Indeed, if the spaces $\{B^n_\om\}$ and $\{B^n_{{\om}}\otimes \ell_2\}$, $n=0, 1, ... ,N$,  are part of  interpolation scales   $\{B^s_\om\}$ and $\{B^s_{{\om}}\otimes \ell_2\}$, $0 \le s\le N, s\in \R$, obtained by the complex method, then the functorial property of the interpolation implies that the hypothesis of Theorem \ref{IntroRowColumnOp} reduces to $M^C(B^N_{\om})\subseteq  M^C(B^0_{\om})=H^\infty(\ell_2)$. See \cite{BerghLoefstroem} for information about the complex method.

Thus, our theorem implies that if the spaces $\{B^N_\om\}_{N\in \N_0}$ and $\{B^N_\om\otimes \ell_2\}_{N\in \N_0}$ are part of  interpolation scales obtained by the complex method, then every bounded column multiplication operator on $B^N_\om$ induces a bounded row operator.

 For standard weights and more generally for weights that satisfy a Bekoll\'{e}-Bonami condition it was shown in \cite{CasFabOrt} and \cite{CasFab_Bekolleweights} that $\{B^s_\om\}_{s\in\R}$ is an interpolation scale, and that the spaces satisfy the scalar version of the multiplier inclusion condition (for standard weights also see \cite{ZhaoZhu}). The full column operator multiplier inclusion condition follows similarly in those cases. In \cite{AlHaMcCRiWeightedB} we similarly show that in fact for every admissible radial measure and every $s\in \R$ the space $B^s_\om$ satisfies the multiplier inclusion condition and the conclusion of Theorem \ref{IntroRowColumnOp} holds. In that paper we also show that if a radial measure $\om$ satisfies that for some $\alpha >-1$ the ratio $\om(z)/(1-|z|^2)^\alpha$ is nondecreasing in $|z|$ for $t_0<|z|<1$, then
 $B^s_\om$ is a complete Pick space, whenever $s\ge (\alpha+d)/2$. By a complete Pick space we mean a Hilbert function space $\HH$ such that there is a norm on $\HH$ that is equivalent to the original one, and such that the reproducing kernel for  one of the norms is a normalized complete Pick kernel.

 As another application of Theorem \ref{IntroRowColumnOp} we mention  that it provides a new proof of  the main result of \cite{AlHaMcCRiInter} in the case where the complete Pick space is a radially weighted Besov space. Indeed, in \cite{AlHaMcCRiInter} the proof of the characterization of the interpolating sequences for all spaces with complete Pick kernel is based on the Marcus-Spielman-Srivastava theorem \cite{MarcusSpielmanSrivastava}, but in Remark 3.7 and Theorem 3.8 of \cite{AlHaMcCRiInter} it is explained how an application of Theorem \ref{IntroRowColumnOp}  provides an alternate proof. We particularly point out that this approach provides a very direct proof for the case of  the Drury-Arveson space on a finite dimensional ball. For this only the results of Section 4 are needed.

\

This paper is organized as follows. In Section \ref{BackgroundWPH} we prove that the weak product always carries a weak* topology such that point evaluations are weak*-continuous, and we review the connection between $(\WPH)^*$ and Hankel operators.  Corollary \ref{Smirnov}  contains Theorem \ref{IntroSmirnov1a}  and the Smirnov class inclusions of Theorem \ref{IntroSmirnov2}, while in Theorem \ref{WeakProductQuotient} we have provided a technical version which gives more information. In Section \ref{SectionMultInvsubspaces} we have proved   Theorem \ref{IntroSmirnov1b} and the parts of Theorem \ref{IntroSmirnov2} that give information about the invariant subspaces. Corollary \ref{HankelIntersection} states that if a complete Pick space $\HH$ satisfies the condition $M^C(\HH)\subseteq  M^R(\HH)$, then all multiplier invariant subspaces of $\HH$ are equal to a countable intersection of null spaces of bounded Hankel operators. This extends results of \cite{LuoRi} and \cite{RiSunkes}. Section 4 is independent of the results of Sections 2 and 3, it contains our results on weighted Besov spaces. Theorem \ref{IntroRowColumnOp} will be a special case of Theorem \ref{RowColumnOp}, where operators between possibly different spaces are considered. In \cite{TrentCorona} Trent has provided an example that shows that $M^R(D) \nsubseteq M^C(D)$.
Since Trent's example  does not immediately generalize from the Dirichlet space to the Drury-Arveson space, we have provided an example of a bounded row multiplication operator on $H^2_d, d>1,$ that does not induce a bounded column operator, see Section 4.2.

\section{Background on the weak product of a Hilbert function space}
\label{BackgroundWPH}
In \cite{RiSuWeakProd} some general results about the weak product $\WPH$ and its dual were shown for the case when $\HH$ is a Hilbert space of analytic functions. In particular, it was shown that $\WPH$ is always a Banach function space and that its dual can be identified with a space of Hankel operators, provided the space $\HH$ satisfied a certain extra hypothesis. This makes it reasonable to conjecture that $\WPH$ has an isometric predual which can be identified with a space of compact Hankel operators. We will now show that indeed for any Hilbert function space $\HH$ the weak product has an isometric predual and if $\Mult(\HH)$ is densely contained in $\HH$, then the dual and predual of $\WPH$ can each be identified with Hankel operators defined by  symbol sets that are contained in $\HH$.

 \subsection{The predual of $\WPH$}

\begin{theorem}\label{predual} Let $\HH$ be a Hilbert function space on a set $X$. Then the weak product $\HH\odot\HH$ is a Banach function space on $X$.
It is isometrically isomorphic to the dual of a Banach space in such a way that point evaluations on $\HH\odot \HH$ are weak* continuous with respect to the duality.
\end{theorem}

\begin{proof}
  Let $\HH \otimes_{\pi} \HH$ denote the Banach space projective tensor product
  of $\HH$ with itself, that is, the completion of the algebraic tensor product with
  respect to the norm
  \begin{equation*}
    ||h|| = \inf \Big \{ \sum_{n=1}^n ||f_n|| \, ||g_n|| : h = \sum_{n=1}^n f_n\otimes g_n \Big \}.
  \end{equation*}
  Every element $u$ of $\HH \otimes_{\pi} \HH$ can in fact be written in the form
  \begin{equation*}
    u = \sum_{n=1}^\infty f_n \otimes g_n \quad \text{ with } \sum_{n=1}^\infty ||f_n|| \, ||g_n|| < \infty,
  \end{equation*}
  see e.g. \cite{RyanBanachTensor}. In the following we will use the Hilbert space of complex conjugates
  $$\overline{\HH}=\{\overline{f}: f\in \HH\},$$ which has inner product given by $\la \overline{f}, \overline{g}\ra=\la g,f\ra$ and can be isometrically identified with the dual $\HH^*$ via the correspondence $\overline{f} \to L_{\overline{f}}$, $L_{\overline{f}}(g)=\la g,f\ra_{\HH}$ for $g\in \HH$.

  By definition of the weak product, the map
  \begin{equation*}
    \rho: \HH \otimes_{\pi} \HH \to \HH \odot \HH, \quad \rho \Big( \sum_{n=1}^\infty f_n \otimes g_n  \Big)
    (z) = \sum_{n=1}^\infty f_n(z) g_n(z)
  \end{equation*}
  is a quotient map. For $z \in X$, let $E_z=\overline{k}_z\otimes \overline{k}_z \in (\HH \otimes_{\pi} \HH)^*$ denote the functional of evaluation at $z$. Then
  \begin{equation*}
    \ker \rho = \bigcap_{z \in X} \ker E_z,
  \end{equation*}
 thus  $\rho$ induces an isometric isomorphism $(\HH \otimes_{\pi} \HH) / \ker\rho \cong \HH \odot \HH$.

  Let $\mathcal C_1(\overline{\HH}, \HH)$ denote the space of all trace class operators from $\overline{\HH}$ to $\HH$.
  Then $\HH \otimes_{\pi} \HH$ can be isometrically identified with $\mathcal C_1(\overline{\HH},\HH)$ via the map
  \begin{equation*}
    \Phi: \HH \otimes_{\pi} \HH \to \mathcal C_1(\overline{\HH},\HH), \quad \Phi(f \otimes g)(\overline{h}) = \la g,h\ra f,
  \end{equation*}
  see \cite{EffrosRuan}, Proposition 8.2.1.
  On the other hand, $\mathcal C_1(\overline{\HH},\HH)$ is the dual space of the space of compact operators
  from $\HH$ to $\overline{\HH}$ via trace duality. Thus, $\HH \otimes_{\pi} \HH$ becomes a dual space in this way,
  and every functional of the form $\overline{f} \otimes \overline{g}$ on $\HH \otimes_{\pi} \HH$ for $\overline{f},\overline{g} \in \overline{\HH}$
  is weak* continuous. In particular, it follows that $\ker\rho$ is weak* closed and
   thus $\HH \odot \HH$ can be identified with the dual of $^\bot \ker \rho$. Since $\overline{k}_z \otimes \overline{k}_z$
  belongs to $^\bot \ker\rho$ for each $z$, point evaluations on $\HH \odot \HH$ are weak* continuous
  with respect to this duality.
\end{proof}

It follows from the Hahn-Banach theorem that the linear span of the point evaluations is dense in the predual of $\HH \odot \HH$.
Using the uniform boundedness principle, we therefore obtain the following standard corollary.
\begin{corollary} \label{weak*convergence} Let $\HH$ be a  Hilbert function space  on a set $X $, and let $h_n\in \HH\odot \HH$ be a sequence of functions. Then the following are equivalent for a function $h$ on $X$:

 (a) $h\in \HH\odot \HH$ and $h_n\to h$ in the weak* topology given by the previous theorem,

 (b) $\|h_n\|_{\HH\odot \HH} \le C$ and $h_n(z) \to h(z)$ for all $z\in X$.
\end{corollary}

We further remark that if $\HH$ is separable, then so are $\WPH$ and its predual. It follows that in this case the closed unit ball of $\WPH$ is compact metrizable in the weak* topology.

 \subsection{The Connection to Hankel operators}\label{SecHankel}
Since the dual space of $\mathcal C_1(\overline{\HH},\HH)$ is the space $\mathcal B(\HH,\overline{\HH})$ via trace duality and since $\HH \otimes_{\pi} \HH \cong \mathcal C_1(\overline{\HH},\HH)$, every $T\in \mathcal B(\HH,\overline{\HH})$ defines a linear functional on $\HH \otimes_{\pi} \HH$ by \begin{equation}\label{functional} f\otimes g \to \la g, \overline{Tf}\ra.\end{equation}
 Let
 \begin{equation*}
   \rho: \HH \otimes_{\pi} \HH \to \HH \odot \HH
 \end{equation*}
 be the quotient map from the proof of Theorem \ref{predual}. Then
\begin{equation*}
  (\HH \odot \HH)^* \cong (\ker \rho)^\bot \subseteq  (\HH \otimes_{\pi} \HH)^* \cong \mathcal B(\HH, \overline{\HH}).
\end{equation*}
We will now see that if the multipliers are dense in $\HH$, then the operators in $(\ker \rho)^\bot$ can be considered to be little Hankel operators, each of which is identified with a symbol from the space ${\HH}$.

\begin{lemma}\label{HankelLemma} If $T\in (\ker \rho)^\bot$, then

(a) $T^* \overline{f}=\overline{Tf}$ for every $f\in \HH$, and

(b) $TM_\varphi= M^*_{\overline{\varphi}}T$ for every $\varphi \in \Mult(\HH)$.

Furthermore, if $\Mult(\HH)$ is densely contained in $\HH$, then for $T\in (\ker \rho)^\bot$ we have $T=0$ if and only if $T1=0$.
\end{lemma}
\begin{proof} Let $T\in (\ker \rho)^\bot$. Then for  $f,g \in \HH$ and $\varphi\in \Mult(\HH)$ we have $f\otimes \varphi g -g \otimes \varphi f \in \ker \rho$, hence $\la \varphi g,\overline{Tf}\ra = \la \varphi  f, \overline{Tg}\ra = \la Tg, \overline{\varphi f}\ra_{\Hbar}= \la g, T^*\overline{\varphi f}\ra$. Thus $T^*\overline{\varphi f}=M^*_{\varphi} \overline{Tf}$, and (a) follows by taking $\varphi=1$.

Next we substitute (a) into $T^* M_{\overline{\varphi}}\overline{f}=T^*\overline{\varphi f}=M^*_{\varphi} \overline{Tf}$, then (b) follows by taking adjoints.

The remaining part of the Lemma follows from (b).
\end{proof}

Thus, if $\Mult(\HH)$ is densely contained in $\HH$, then an operator  $T\in (\ker \rho)^\bot$ is uniquely associated with the function  $T1$, and $T$
intertwines multiplication operators and adjoints of multiplication operators, hence $T$ deserves to be called a Hankel operator.

\begin{definition} Let $\HH$ be a Hilbert function space such that $\Mult(\HH)$ is densely contained in $\HH$. Then define
\begin{equation*}
  \Han(\HH) = \{ \overline{ T 1} \in {\HH}: T \in (\ker \rho)^{\bot}\}.
\end{equation*}
For $b\in \Han(\HH)$ we write $H_b\in \HB(\HH, \Hbar)$ for the unique operator in $(\ker \rho)^\bot$ that satisfies $H_b1=\overline{b}$, and we set $\|b\|_{\Han(\HH)}= \|H_b\|$.

Furthermore, we define $$\Han_0(\HH)=\{b\in \Han(\HH): H_b \text{ is compact }\}.$$
\end{definition}
With these definitions we have that $\Han_0(\HH)$ is isometrically isomorphic to $^\bot(\ker \rho)$, and  the following Theorem holds.

\begin{theorem}  \label{theo:han_basic}
Let $\HH$ be a Hilbert function space such that $\Mult(\HH)$ is densely contained in $\HH$. Then  the following conjugate linear isometric isomorphisms hold:

(a) $\Han_0(\HH)^* \cong \WPH$ and

(b) $(\WPH)^* \cong \Han(\HH)$.

If $b\in \Han(\HH)$, then the associated linear functional $L_b$ satisfies $$L_b(\varphi f)=\la \varphi f,{b}\ra = \la f, \overline{H_b \varphi}\ra= \la \varphi, \overline{H_b f}\ra $$ for every $f\in \HH$ and $\varphi\in \Mult(\HH)$.
\end{theorem}
\begin{proof} We have explained the isometric isomorphisms above. Let $f\in \HH$ and $\varphi\in \Mult(\HH)$.
According to (\ref{functional}) and the definition of $H_b$ we have $L_b(\varphi f)=\la f , \overline{H_b \varphi}\ra= \la \varphi, \overline{H_b f}\ra$. But then Lemma \ref{HankelLemma} implies
$$L_b(\varphi f)=\la f, H_b^* \overline{\varphi} \ra=  \la f, M^*_{\varphi}H_b^*1\ra= \la \varphi f, \overline{H_b1}\ra = \la \varphi f, {b}\ra.$$
\end{proof}

Theorem \ref{theo:han_basic} does not address the question of how one can easily identify which functions are in $\Han(\HH)$. The set
\begin{equation*}
  \mathcal X(\HH) =
    \{ {b} \in {\HH}: \exists C \ge 0 |\langle \varphi f,b \rangle| \le C ||\varphi||_{\HH} \, ||f||_{\HH} \forall f \in \HH, \varphi \in \Mult(\HH) \}
\end{equation*}
seems more suited for that question, and Theorem \ref{theo:han_basic} implies that   $\Han(\HH)\subseteq \mathcal X(\HH)$.

We note that Theorems 1.2 and 1.3 of \cite{RiSuWeakProd}  imply that $\Han(B^s_\om) = \mathcal X(B^s_\om)= (B^s_\om\otimes B^s_\om)^*$ for all admissible radial weights $\om$ and all $s\in \R$.
Using the main result of \cite{AlHMcCr_Factors}, we will now show that the equality $\Han(\HH) = \mathcal X(\HH)$ also
holds whenever $\HH$ is a complete Pick space with $M^C(\HH) \subseteq M^R(\HH)$.

\begin{theorem}
 Let $\HH$ be a separable Hilbert function space on the non-empty set $X$, and suppose that  the reproducing kernel for $\HH$ is a complete Pick kernel, which is normalized at a point $z_0\in X$.

 If $M^C(\HH) \subseteq M^R(\HH)$, then $\Han(\HH) = \mathcal X(\HH)$.
\end{theorem}

\begin{proof}
  Since $\HH$ has a normalized complete Pick kernel, $\Mult(\HH)$ is densely contained in $\HH$.
  As mentioned above, by Theorem \ref{theo:han_basic}   it suffices to show that $\mathcal X(\HH) \subseteq \Han(\HH)$.
  Let ${b} \in \mathcal X(\HH)$.
  In order to show that ${b} \in \Han(\HH)$, we note that
  the definition of $\mathcal X(\HH)$ and the universal property of the projective tensor product show that there exists
  a bounded linear functional $L$ on $\HH \otimes_{\pi} \HH$ such that
  \begin{equation*}
    L(f \otimes \varphi) = \langle f \varphi,b \rangle
  \end{equation*}
  for all $f \in \HH$ and $\varphi \in \Mult(\HH)$. We claim that $L \in (\ker\rho)^\bot$.
  Assuming this claim for a moment, we can regard $L$ as a functional on $\HH \odot \HH$, hence
  by Theorem \ref{theo:han_basic},
  there exists ${c} \in \Han(\HH)$ such that
  \begin{equation*}
    \langle \varphi f, c \rangle = L(\varphi f) = \langle \varphi f, b  \rangle
  \end{equation*}
  for all $f \in \HH, \varphi \in \Mult(\HH)$, so that ${b} = {c} \in \Han(\HH)$.

  To prove the claim, let $h \in \ker\rho$ with $||h||_{\HH \otimes_\pi \HH} < 1$.
  We wish to show that $L(h) = 0$. To this end, observe that there
  exist $f_n,g_n \in \HH$ with $||f_n|| = ||g_n||$ for all $n$,
  \begin{equation*}
    h = \sum_{n=1}^\infty f_n \otimes g_n
  \end{equation*}
  and $\sum_{n=1}^\infty ||f_n||^2 = 1$.
  By \cite[Theorem 1.1]{AlHMcCr_Factors},
  there exist $\psi, \varphi_n \in \Mult(\HH)$ such that $\psi(z_0) =0$,
  \begin{equation*}
    || \psi h ||^2 + \sum_n ||\varphi_n h||^2 \le ||h||^2 \quad (h \in \HH)
  \end{equation*}
  and $f_n = \frac{\varphi_n}{1 - \psi}$ for all $n \in \mathbb N$. For $r \in (0,1)$, let
  \begin{equation*}
    f_n^{(r)} = \frac{\varphi_n}{1 -r \psi}.
  \end{equation*}
  Then $[f_n^{(r)}]_{n=1}^\infty \in \Mult(\HH, \HH(\ell^2))$ for each $r<1$ by \cite[Lemma 3.6 (i)]{AlHMcCr_Factors}.
  Moreover, by the remark at the end of Section 3 of \cite{AlHMcCr_Factors}, $[f_n^{(r)}]$
  converges to $[f_n]$ in the norm of $\HH(\ell^2)$ as $r \to 1$. Thus,
  \begin{align*}
    &\Big\| \sum_n f_n^{(r)} \otimes g_n - \sum_n f_n \otimes g_n \Big\|^2_{\HH \otimes_\pi \HH} \\
    \le &\Big( \sum_n ||f_n^{(r)} - f_n||^2 \Big)
    \Big( \sum_n ||g_n||^2 \Big) \xrightarrow{r \to 1} 0,
  \end{align*}
  so by continuity of $L$, it suffices to show that
  \begin{equation*}
    L \Big( \sum_{n} f_n^{(r)} \otimes g_n \Big) = 0
  \end{equation*}
  for all $r \in (0,1)$.

  To see this, fix $r \in (0,1)$.
  The series $\sum_n f_n^{(r)} \otimes g_n$ converges absolutely in the norm of $\HH \otimes_{\pi} \HH$, so that
  \begin{equation*}
    L \Big( \sum_{n} f_n^{(r)} \otimes g_n \Big) = \sum_{n} L(f_n^{(r)} \otimes g_n)
    = \sum_n \langle f_n^{(r)} g_n, b \rangle.
  \end{equation*}
  Since $[f_n^{(r)}] \in \Mult(\HH,\HH(\ell^2))$, the assumption $M^C(\HH) \subseteq M^R(\HH)$
  implies that the series $\sum_{n} f_n^{(r)}  g_n$
  converges in $\HH$, so that
  \begin{equation*}
    L \Big( \sum_{n} f_n^{(r)} \otimes g_n \Big) = \Big \langle \sum_n f_n^{(r)} g_n, b \Big \rangle = 0,
  \end{equation*}
  where the last equality follows from the observation that
  since $h \in \ker \rho$, we also have
  \begin{equation*}
    \sum_{n=1}^\infty f_n^{(r)}(z) g_n(z) = \frac{1 - \psi(z)}{1 - r \psi(z)} \sum_{n=1}^\infty f_n(z) g_n(z) = 0
  \end{equation*}
  for all $z \in X$.
\end{proof}

\section{Weak products of complete Pick spaces}
\label{WeakProdPick}
\subsection{Functions as ratios of multipliers}

\begin{theorem}\label{RowColumnMultiplier} Let $\HH$ be a separable Hilbert function space on a set $X$ with reproducing kernel $k_z\ne 0$ for all $z\in X$, and let $\{\Phi_n^C\}_{n\ge 1}, \{\Psi_n^C\}_{n\ge 1}\subseteq  M^C(\HH)$ be sequences of column operators such that
$$\sum_{n\ge 1}\|\Phi_n^C f\|^2_{\HH\otimes \ell_2} \le \|f\|^2_\HH \text{ and } \sum_{n\ge 1}\|\Psi_n^C f\|^2_{\HH\otimes \ell_2} \le \|f\|^2_\HH$$ for all $f\in \HH$.

 Then for each $n\in \N$ we have  $\Psi_n^R \Phi_n^C \in \Mult(\HH\odot\HH)$ and  $$\sum_{n\ge 1}\|\Psi_n^R \Phi_n^Ch\|_{\WPH} \le \|h\|_{\WPH}$$ for all $h\in \WPH$.

Furthermore,  if $\HH$ satisfies the continuous inclusion $M^C(\HH)\subseteq  M^R(\HH)$,  then  for each $n\in \N$ we have  $ \Psi_n^R \Phi_n^C\in \Mult(\HH)$ and
there is a $c>0$ such that for all   $f\in \HH$ we have
 $$\sum_{n\ge 1} \|\Psi_n^R \Phi_n^Cf\|_\HH^2\le c\|f\|_\HH^2 \text{ and }$$ $$\|\sum_{n=1}^N \Psi_n^R \Phi_n^Cf\|^2_\HH\le c\|f\|^2_\HH \ \text{ for each }N\in \N.$$
\end{theorem}
\begin{proof} For each $n\in \N$ let $\Phi_n=\{\varphi_{ni}\}_{i\ge 1}, \Psi_n=\{\psi_{ni}\}_{i\ge 1}$ define bounded column operators that satisfy the hypothesis of the Theorem. Then for any $f,g \in \HH$ we have
  $$\sum_{i\ge 1}\|\varphi_{ni} \psi_{ni} fg\|_{\HH\odot\HH} \le \sum_{i\ge 1}\|\varphi_{ni} f\|\|\psi_{ni} g\| \le \|\Phi_n^C f\|\|\Psi_n^C g\|.$$
It is well-known and easy to show that $\|\Phi_n(z)\|_{\ell_2}\le \|\Phi_n^C\|_{M^C(\HH)}$ and $\|\Psi_n(z)\|_{\ell_2} \le  \|\Psi_n^C\|_{M^C(\HH)}$ for each $z \in X$. Thus $\Psi_n^R(z) \Phi_n^C(z)= \sum_{i\ge 1}\psi_{ni}(z)\varphi_{ni}(z)$ converges absolutely, $(\Psi_n^R(z) \Phi_n^C(z))f(z)g(z)= \sum_{i\ge 1} (\varphi_{ni} f)(z)(\psi_{ni} g)(z)$,
and hence
\begin{align*}\sum_{n\ge 1}\|(\Psi_n^R \Phi_n^C)fg\|_{\HH\odot\HH} &\le \sum_{n,i\ge 1}\|\varphi_{ni} \psi_{ni} fg\|_{\HH\odot\HH} \\
  &\le \sum_{n\ge 1} \|\Phi_n^Cf\|_{\HH \otimes \ell^2}\|\Psi_n^Cg\|_{\HH \otimes \ell^2}\\
&\le \|f\|_\HH\|g\|_\HH.\end{align*} Let $h\in \HH\odot \HH$ and let
 $\{f_j\},\{g_j\} \in \bigoplus_{j = 1}^\infty \HH$ with $h= \sum_{j=1}^\infty f_jg_j$. Then
 \begin{align*} \sum_{n\ge 1}\|(\Psi_n^R \Phi_n^C)h\|_{\HH\odot\HH} &\le \sum_{n,j=1}^\infty \|(\Psi_n^R \Phi_n^C)f_jg_j\|_{\HH\odot\HH} \\
 &\le \sum_{j=1}^\infty \|f_j\|_\HH\|g_j\|_\HH.
\end{align*} Taking the infimum over all possible representations $h= \sum_{j=1}^\infty f_jg_j$, we obtain $\sum_{n\ge 1}\|(\Psi_n^R \Phi_n^C)h\|_{\HH\odot\HH} \le \|h\|_{\WPH}$.

If every bounded column operator on $\HH$ induces a bounded row operator and the inclusion has norm $\sqrt{c}$, then since for each $n$ we have $\|\Psi^C_n\|\le 1$ it follows  easily that $\Psi_n^R \Phi_n^C \in \Mult(\HH)$ with $\|\Psi_n^R \Phi_n^C f\|^2 \le c\|\Phi_n^C f\|^2_{\HH \otimes \ell_2}$. Thus, an application of the hypothesis finishes the proof of the Theorem.\end{proof}

It was shown in \cite{AlHaMcCRiM}, see also \cite{AlHMcCr_Factors},
that if $\mathcal H$ is a Hilbert function space on $X$ with a
complete Pick kernel, normalized at a point $z_0 \in X$,
then every $f \in \mathcal H$ can be written as $f = \frac{\varphi}{1- \psi}$,
where $\varphi,\psi \in \Mult(\mathcal H )$, $||\psi||_{\Mult(\mathcal H)} \le 1$ and $\psi(z_0) = 1$. In this case, $|\psi(z)| < 1$ for all $z \in X$, hence $\frac{\varphi}{1- \psi}$ is defined on $X$, see Lemma 2.2 of \cite{AlHaMcCRiM}.
\cite{AlHMcCr_Factors} also contains a vector version of this result.
The following
lemma is an analogue of this result for the weak product $\mathcal H \odot \mathcal H$.

\begin{lemma}\label{WeakProductQuotientLemma}  Let $\HH$ be a separable Hilbert function space on the non-empty set $X$, and suppose that  the reproducing kernel for $\HH$ is a complete Pick kernel, which is normalized at a point $z_0\in X$.

If  $\{h_n\}_{n\ge 1}\subseteq   \HH \odot \HH$, $\sum_{n \ge 1}\|h_n\|_{\HH\odot\HH}< 1$,  then  there are $\psi \in \Mult(\HH)$ and $\{\Phi_n^C\}_{n\ge 1} \subseteq  M^C(\HH)$ such that

 (a)  $\|\psi\|_{\Mult(\HH)}\le 1$ and $\psi(z_0)=0$,

 (b)  $\sum_{n\ge 1}\|\Phi_n^Cu\|^2_{\HH\otimes \ell_2}\le \|u\|_\HH^2$ for all $u \in \HH$,

 (c) for each $n \in \N$ $h_n=\frac{\varphi_n}{(1-\psi)^2}$  with $\varphi_n=\Phi_n^R\Phi_n^C\in \Mult(\HH\odot \HH)$ and
 $$\sum_{n\ge 1}\|\varphi_nh\|_{\WPH}\le \|h\|_{\WPH} \text{ for all }h\in \WPH.$$

 (d) If additionally it is true that $M^C(\HH)\subseteq  M^R(\HH)$, then
 there is a $c>0$ such that
 $$\sum_{n\ge 1} \|\varphi_nf\|^2_\HH\le c \|f\|^2_\HH$$ for each $f\in \HH$. Furthermore, for each an $N\in \N$ we have $$\|\sum_{n=1}^N \varphi_n\|_{\Mult(\HH)} \le c.$$
\end{lemma}

\begin{proof}
Note that if $f,g \in \HH$ with $\|f\|=\|g\|$, then $fg= \left(\frac{f+g}{2}\right)^2 - \left(\frac{f-g}{2}\right)^2$ with $\|f\| \|g\|= \|\frac{f+g}{2}\|^2+\|\frac{f-g}{2}\|^2$ by the parallelogram law. Thus, for any $h\in \WPH$ and any $\varepsilon >0$ there are $f_j\in \HH$ with $\sum_{j\ge 1}\|f_j\|_\HH^2 \le \|h\|_\WPH +\varepsilon$ and
$h=\sum_{j\ge 1} f_j^2$. Hence for each $n$ we  choose a sequence $\{f_{nj}\}_{j\ge 1} \subseteq  \HH$ such that $h_n= \sum_{j\ge 1} f_{nj}^2$ and $\sum_{n,j\ge 1}\|f_{nj}\|_\HH^2 \le 1.$

 Then by Theorem 1.1 of \cite{AlHMcCr_Factors}, there
there are contractive multipliers $\psi$, $\{\varphi_{nj}\}_{n,j\ge 1}$ such that $\psi(z_0)=0$, $\|\psi g\|^2 + \sum_{n,j\ge 1}\|\varphi_{nj}g\|^2 \le \|g\|^2$ for all $g \in \HH$, and $f_{nj}= \frac{\varphi_{nj}}{1-\psi} $ for all $n, j \ge 1$. Then  $h_n=\frac{\varphi_n}{(1-\psi)^2}$ with $\varphi_n= \sum_{j \ge 1}\varphi_{nj}^2=\Phi_n^R\Phi_n^C$, where $\Phi_n=\{\varphi_{nj}\}_{j\ge 1}$.
Thus the lemma follows from Theorem \ref{RowColumnMultiplier} with $\Psi_n=\Phi_n$.
\end{proof}
The following Theorem is  a slight refinement of the previous Lemma in the case of a single function.

\begin{theorem}\label{WeakProductQuotient} Let $\HH$ be a separable Hilbert function space on the non-empty set $X$, and suppose that  the reproducing kernel for $\HH$ is a complete Pick kernel, which is normalized at a point $z_0\in X$.

If $h\in \HH \odot \HH$, then  there are $\varphi \in \Mult(\HH \odot \HH)$, $\|\varphi\|_{\Mult(\HH \odot \HH)}\le  \|h\|_{\HH\odot\HH}$ and $\psi \in \Mult(\HH)$, $\|\psi\|_{\Mult(\HH)}\le 1$, $\psi(z_0)=0$ such that $h=\frac{\varphi}{(1-\psi)^2}$.
\end{theorem}
\begin{proof} We assume $\|h\|_{\HH\odot\HH}=1$.  For each $m\in \N$ we apply the single function version of Lemma \ref{WeakProductQuotientLemma} with $(1-\frac{1}{m+1})h$ and thus obtain functions $\varphi_m,\psi_m$ with $\|\psi_m\|_{\Mult(\HH)} \le 1$, $\psi_m(z_0)=0$,  $\|\varphi_m\|_{\Mult(\HH\odot\HH)} \le 1$, and $h=(1+1/m)\frac{\varphi_m}{(1-\psi_m)^2}$. It follows from the hypothesis that $\HH$ and $\HH \odot \HH$ are separable, thus we can assume that there are  subsequences such that $\varphi_{m_j}\to \varphi$ in the weak* topology of $\HH \odot \HH$ and $\psi_{m_j}\to \psi$ weakly in $\HH$.

  Weak* and weak convergence imply pointwise convergence, and hence the norm bound and Corollary \ref{weak*convergence} imply $\varphi_{m_j}g \to \varphi g$ weak* in $\HH \odot \HH$ for each $g\in \HH \odot \HH$.
  Since this is valid for all $g\in \HH\odot \HH$  we conclude $\|\varphi\|_{\Mult(\HH\odot\HH)}\le 1$. Similarly $\|\psi\|_{\Mult(\HH)} \le 1$. Since $\psi(z_0)=0$ Lemma 2.2 of \cite{AlHaMcCRiM} implies that $|\psi(z)|<1$ for all $z\in X$. Thus $\varphi/(1-\psi)^2$ is well-defined and in that case it clearly must equal $h$.
\end{proof}

\begin{corollary} \label{Smirnov} Let $\HH$ be a separable Hilbert function space on the non-empty set $X$, and suppose that  the reproducing kernel for $\HH$ is a normalized complete Pick kernel. Then $$\HH\odot \HH \subseteq  N^+(\WPH).$$

If additionally $M^C(\HH)\subseteq  M^R(\HH)$, then
$$\HH\odot \HH \subseteq  N^+(\HH).$$
\end{corollary}

\begin{proof} Suppose the reproducing kernel is normalized at the point $z_0\in X$.
  Let $h\in \HH\odot \HH$. By Lemma \ref{WeakProductQuotientLemma} we have  $h=\frac{\varphi}{(1-\psi)^2}$ for  $ \psi\in \Mult(\HH)$, $\Phi^C\in M^C(\HH)$ with $\|\psi\|_{\Mult(\HH)} \le 1$, $\psi(z_0)=0$, and $\varphi= \Phi^R\Phi^C \in \Mult(\WPH)$ by Theorem \ref{RowColumnMultiplier}.

It now follows immediately from  Lemma 2.3 of \cite{AlHaMcCRiM} that $1-\psi$ is cyclic in $\HH$. Furthermore, it is  easy to see that products of cyclic multipliers are cyclic, hence the corollary  follows.
\end{proof}

%-------------------------------------------------------------------------------------------------------------------------------------------------------------------

\subsection{Multiplier Invariant subspaces}
\label{SectionMultInvsubspaces}
Recall that if $\HM $ is a closed subspace of a Banach function space $\HB$, then we say $\HM$ is multiplier invariant, if $\varphi \HM \subseteq  \HM$ for all multipliers $\varphi \in \Mult(\HB)$. If $f\in \HB$, then we write $[f]_\HB$ for the smallest multiplier invariant subspace of $\HB$ that contains $f$, i.e.
 $$[f]_\HB=\clos_{\HB}\{\varphi f: \varphi \in \Mult(\HB)\}.$$ Thus, since $\Mult(\HH) \subseteq  \Mult(\WPH)$ we have
 $$\clos_{\WPH}\{\varphi f: \varphi \in \Mult(\HH)\}\subseteq  [f]_\WPH.$$ We will start this section by showing that for complete Pick spaces $\HH$ with $M^C(\HH)\subseteq  M^R(\HH)$ these two sets are always the same.

\begin{lemma} \label{weakProdInvSubspaceLemma} Let $\HH$ be a separable Hilbert function space on a set $X$.
  If $h_1,h_2 \in \HH\odot \HH$  and $ \psi_n \in \Mult(\HH)$ with
\begin{enumerate}\item $\psi_nh_2\in \clos_{\WPH} \{uh_1 : u \in \Mult(\HH)\}$ for each $n$,
 \item $\psi_n(z)\to 1 $ for each $z\in X$ and
\item $\|\psi_n \|_{\Mult(\HH)} \le C$ for each $n$,
\end{enumerate}
then $h_2\in \clos_{\HH\odot\HH}\{uh_1: u\in \Mult(\HH)\}\subseteq  [h_1]_\WPH$.
\end{lemma}
\begin{proof}
  Let $M$ be the convex hull of $\{\psi_n : n \in \mathbb N \}$ inside of $\Mult(\HH)$.
  It follows from assumptions (ii) and (iii) that $1$ belongs to the WOT-closure of $M$.
  By convexity of $M$, there is a sequence $(\varphi_n)$ in $M$ that converges to $1$ in the strong
  operator topology of $\HH$. It is then straightforward to check that
  the sequence $(\varphi_n h_2)$ converges to $h_2$ in the norm of $\HH \odot \HH$,
  so assumption (i) implies that $h_2 \in \operatorname{clos}_{\HH \odot \HH} \{ u h_1: u \in \Mult(\HH) \}$.
\end{proof}

\begin{lemma}\label{cyclic subspaces} Let $\HH$ be a separable Hilbert function space on a set $X$, and let  $z_0\in X$.

(a) If $f, g\in \HH$,  $ \psi \in \Mult(\HH)$, $\|\psi\|_{\Mult(\HH)}\le 1$, $\psi(z_0)=0$ such that $f=\frac{g}{1-\psi}$, then $[f]_\HH=[g]_\HH$.

(b) If $h,g\in \HH \odot \HH$,   $\psi \in \Mult(\HH)$, $\|\psi\|_{\Mult(\HH)}\le 1$, $\psi(z_0)=0$ such that $h=\frac{g}{(1-\psi)^2}$,
then $h\in \clos_{\HH\odot \HH}\{ug: u \in \Mult(\HH)\}$ and $[h]_{\HH\odot \HH}=[g]_{\HH\odot \HH}$.
\end{lemma}
\begin{proof} (a) Let $f, g\in \HH$,  $ \psi \in \Mult(\HH)$, $\|\psi\|_{\Mult(\HH)}\le 1$, $\psi(z_0)=0$ such that $f=\frac{g}{1-\psi}$. Then $g=(1-\psi)f \in [f]_\HH$. Thus $[g]_\HH\subseteq  [f]_\HH$.

Let $0<r<1, $ then $1/(1-r\psi)\in \Mult(\HH)$ and $\frac{g}{1-r\psi}\in [g]_\HH$. A short calculation shows that
$\|\frac{g}{1-r\psi}-f\|=\|\frac{(1-r)\psi}{1-r\psi}f\|\le \|f\|.$
Thus, it follows that $\frac{g}{1-r\psi}$ converges weakly to $f$ as $r\to 1^-$. Hence $f\in [g]_\HH$. This proves (a).

(b) Let $h, g\in \HH \odot \HH$,  $\psi \in \Mult(\HH)$, $\|\psi\|_{\Mult(\HH)}\le 1$, $\psi(z_0)=0$ such that $h=\frac{g}{(1-\psi)^2}$. As in (a) the inclusion $g=(1-\psi)^2h\in [h]_\WPH$ is trivial. For $0<r<1$ we have $\|\frac{1-\psi}{1-r\psi}\|_{\Mult(\HH)}=\|1-\frac{(1-r)\psi}{1-r\psi}\|_{\Mult(\HH)}\le 2$, hence (b) now follows from Lemma \ref{weakProdInvSubspaceLemma} with $\psi_n=\left(\frac{1-\psi}{1-r_n\psi}\right)^2$ for $r_n\to 1^-$ since $\psi_nh= \frac{g}{(1-r_n\psi)^2}$ and $\frac{1}{(1-r_n\psi)^2}\in \Mult(\HH)$ for each $n$.
\end{proof}

\begin{theorem} \label{InvSub}
Let $\HH$ be a separable Hilbert function space on the non-empty set $X$, and suppose that  the reproducing kernel for $\HH$ is a complete Pick kernel, which is normalized at a point $z_0\in X$.

(a) Then
 $\HM\cap \Mult(\WPH)$ is dense in $\HM$ for every multiplier invariant subspace $\HM$ of $\WPH$.
 If $\Mult(\WPH)$ has no zero-divisors, the lattice of  multiplier invariant subspaces of $\WPH$ is cellularly indecomposable, i.e. whenever $\HM, \HN$ are such invariant subspaces with $\HM \ne (0), \HN \ne (0)$, then $\HM\cap \HN \ne (0)$.

(b) If additionally $M^C(\HH)\subseteq  M^R(\HH)$, then  every $\Mult(\HH)$-invariant subspace of $\HH\odot \HH$ is $\Mult(\HH\odot \HH)$-invariant and the map \begin{align*}&\eta:\HN \to \clos_{\WPH} \HN\end{align*}  establishes a  1-1 and onto correspondence between the multiplier invariant subspaces of $\HH$ and of $\WPH$. We have
\begin{enumerate}
\item $\HM=\clos_{\WPH}(\HM\cap \Mult(\HH))$ for every multiplier invariant subspace $\HM$ of $\WPH$, and
\item $\HN=\HH \cap \clos_{\WPH} \HN$ for every multiplier invariant subspace $\HN$ of $\HH$.
\end{enumerate}
\end{theorem}
It is clear from (i) and (ii) that   $\eta^{-1}(\HM)=\HH\cap \HM$. Furthermore, it is easy to see that $\eta$ preserves spans and intersections.
 We note that it follows from Corollary 2.7 of \cite{DavidsonRamseyShalit2015} and Corollary 5.3 of \cite{DavidsonHamilton2011} that for a normalized complete Pick kernel the weak* closed ideals of $\Mult(\HH)$ are in 1-1 and onto correspondence with the multiplier invariant subspaces of $\HH$. An alternate proof of this fact is in \cite{AlRiSunkesCyclicVector}, and the current proof of (ii) is inspired by that approach.

\begin{proof}
(a) Let $\HM$ be a multiplier invariant subspace of $\WPH$, and let $h\in \HM$. Then by Lemma \ref{WeakProductQuotientLemma} $h=\varphi/(1-\psi)^2$ for some $\varphi\in \Mult(\WPH)$ and $\psi\in \Mult(\HH)$ with $\psi(z_0)=0$. Then by Lemma \ref{cyclic subspaces} we have $h\in [\varphi]_\WPH=[h]_\WPH\subseteq  \HM$ and hence there is a sequence $u_n \in \Mult(\WPH)$ such that $u_n\varphi\to h$. Clearly $u_n\varphi\in \HM\cap \Mult(\WPH)$. This proves the first part of (a) and the second part of (a) easily follows from this.

(b) Now suppose that
 $M^C(\HH)\subseteq  M^R(\HH)$. In order to show that  every $\Mult(\HH)$-invariant subspace of $\HH\odot \HH$ is $\Mult(\HH\odot \HH)$-invariant it suffices to   take $h\in \HH\odot \HH$ and  $u\in \Mult(\HH\odot \HH)$ and show that $uh\in \clos_{\HH\odot \HH} \{vh:v\in \Mult(\HH)\}$.

 Since the reproducing kernel of $\HH$ is normalized, $\HH$ contains the constant functions, so we must have $u\in \HH\odot \HH$ and hence by the hypothesis and Lemma \ref{WeakProductQuotientLemma} $u=\frac{\varphi}{(1-\psi)^2}$ for some $\varphi, \psi\in \Mult(\HH)$,  $\|\psi\|_{\Mult(\HH)}\le 1$ and $\psi(z_0)=0$. Then $uh= \frac{\varphi h}{(1-\psi)^2}$ and Lemma \ref{cyclic subspaces} (b) implies that
 $$uh \in \clos_\WPH \{v\varphi h: v\in \Mult(\HH)\}\subseteq  \clos_\WPH \{v h: v\in \Mult(\HH)\}.$$ This establishes the first part of (b). Furthermore, we note that if $\HM$ is a  multiplier invariant subspace of $\WPH$, then $\HH \cap \HM$ is a multiplier invariant subspace of $\HH$ and
 $$\clos_{\WPH}(\HM\cap \Mult(\HH)) \subseteq  \clos_{\WPH}(\HM\cap \HH) \subseteq  \HM.$$ Thus, statement (i) will show that $\eta$ is onto and statement (ii) will show it is 1-1.

 (i) The fact that  $\HM\cap \Mult(\HH)$ is dense in $\HM$ for every multiplier invariant subspace $\HM$ of $\WPH$ follows as in (a), except that now any $h\in \WPH$ is of the form $h=\frac{\varphi}{(1-\psi)^2}$ with $\varphi\in \Mult(\HH)$. This proves (i).

 (ii) Let $\HN$ be a multiplier invariant subspace of $\HH$, we have to show that $\HH \cap \clos_{\WPH} \HN \subseteq  \HN$. To this end let $f_n \in \HN$ and $ f\in \HH$ with $f_n\to f$ in $\WPH$. We have to show that $f\in \HN$. By possibly considering a subsequence we may assume that $\sum_{n\ge 1}\|f_{n+1}-f_n\|_{\WPH} <1$. Now we apply Lemma \ref{WeakProductQuotientLemma} with $h_n=f_{n+1}-f_n$. Thus there are $\psi \in \Mult(\HH)$ and $\{\Phi_n^C\}_{n\ge 1} \subseteq  M^C(\HH)$ such that

 (a)  $\|\psi\|_{\Mult(\HH)}\le 1$ and $\psi(z_0)=0$,

 (b)  $\sum_{n\ge 1}\|\Phi_n^Cu\|^2_{\HH\otimes \ell_2}\le \|u\|_\HH^2$ for all $u \in \HH$,

 (c) for each $n \in \N$ $h_n=\frac{\varphi_n}{(1-\psi)^2}$  with $\varphi_n\in \Mult( \HH)$ and
 $$\|\sum_{n= 1}^N \varphi_n\|_{\Mult(\HH)}\le c$$ for each $N\in \N$.

 Set $g_1=(1-\psi)(f-f_1)$ and $g_2=(1-\psi)g_1$. Then $g_1,g_2\in \HH$ and by Lemma \ref{cyclic subspaces} (a) it suffices to prove that $g_2\in \HN$. But $g_2=\sum_{n\ge 1} \varphi_n$ with $\sum_{n=1}^N\varphi_n=\sum_{n=1}^N(1-\psi)^2(f_{n+1}-f_n)\in \HN$ for each $N$. Since $1\in \HH$ condition (c) from above implies that the partial sums $\sum_{n=1}^N\varphi_n$ converge weakly in $\HH$ to $g_2$. Thus $g_2\in \HN$.
\end{proof}
The following Corollary was known for the Dirichlet space $D$ of the unit disc and Drury-Arveson space $H^2_d$ of the finite dimensional ball $\Bd$, see \cite{LuoRi}, \cite{RiSunkes}.
\begin{corollary}\label{HankelIntersection} Let $\HH$ be a separable Hilbert function space on the non-empty set $X$, and suppose that  the reproducing kernel for $\HH$ is a complete Pick kernel, which is normalized at a point $z_0\in X$.

If  $M^C(\HH)\subseteq  M^R(\HH)$, then for every multiplier invariant subspace $\HM$ of $\HH$, there is a sequence $\{b_n\}$ of symbols of bounded Hankel operators such that $$\HM=\bigcap_n \ker H_{b_n}.$$
\end{corollary}
\begin{proof} Since we are assuming that $\HH$ has a normalized complete Pick kernel it follows that $\Mult(\HH)$ is dense in $\HH$. Thus, by Theorem \ref{theo:han_basic} the dual of $\WPH$ can be identified with $\Han(\HH)$, the set of symbols of bounded Hankel operators $\HH\to \overline{\HH}$. The duality is given by the inner product of $\HH$, and we have
$$\la \varphi f, b\ra = \la \varphi, \overline{H_b f}\ra \ \text{ for all }f\in \HH, \varphi\in \Mult(\HH), b \in \Han(\HH).$$

Let $\HN=\clos_{\WPH} \HM\subseteq  \WPH$ and consider the annihilator $\HN^\perp$ of $\HN$ in $\Han(\HH)$. Since $\Han(\HH)\subseteq  \HH$ and because of the particular form of the duality, it is easy to see that $\HN^\perp \subseteq  \HM^\perp$. Furthermore, if $f\in \HH$ with $f\perp \HN^\perp$, then using again the particular form of the duality
and the Hahn-Banach theorem, we find that $f \in \HN$, thus $f\in \HN \cap \HH=\HM$ by Theorem \ref{InvSub}. Hence $\HN^\perp $ is dense in $\HM^\perp$ in the topology of $\HH$,
and hence there is a countable set $\{b_n\}\subseteq  \HN^\perp$ such that $\{b_n\}$ is dense in $\HM^\perp$.

We claim that $\HM=\bigcap_n \ker H_{b_n}.$ If $f\in \HM$, then for each $\varphi\in \Mult(\HH)$ we have $\la \varphi, \overline{H_{b_n}f}\ra=\la \varphi f, b_n\ra =0$ for each $n$ since $b_n \in \HN^\perp \subseteq  \HM^\perp$. Thus $H_{b_n}f=0$ for each $n$, and hence $\HM\subseteq  \bigcap_n \ker H_{b_n}.$

Note that $H^*_b 1=b$ for every $b\in \Han(\HH)$. Thus,  by the choice of the $b_n$'s $$\HM^\perp =  \bigvee_{n} \{b_n\}\subseteq  \bigvee_{n} \ran H_{b_n}^*=\left(\bigcap_{n} \ker H_{b_n}\right)^\perp.$$ This concludes the proof of the Corollary.
\end{proof}

\section{Column operators between weighted Besov spaces}
\label{SectionWeightedBesov}
\subsection{Multiplier estimates for weighted Besov spaces}

Let $\om$ be an arbitrary admissible  weight, and let $N\in \N$. The admissibility of $\om$ implies that $L^2_a(\om)=B^0_\om$ is a Hilbert space of analytic functions on $\Bd$.
By  use of the identity $f(z)= f(0)+\int_0^1 Rf(tz) \frac{dt}{t}$ one shows that there is an absolute constant $C>0$ such that $|f(z)|\le |f(0)|+ C \sup_{\lambda\in \C,|\lambda|\le 1}|Rf(\lambda z)|$ for any $f\in \Hol(\Bd)$. With this estimate one easily establishes that  each $B^N_\om$ is also a Hilbert function space on  $\Bd$, whenever $\om$ is admissible.

Then in order to check whether an analytic function $\varphi\in \Mult(B^N_\om)$ we must check that there is $C>0$ with $\int_{\Bd} |R^N(\varphi f)|^2 \om dV\le C\|f\|^2_{B^N_\om}$ for all $f\in B^N_\om$. By the Leibnitz rule for the $n$th derivative of a product and the triangle inequality we have
\begin{align} \label{LeibnitzInequality}\int_{\Bd} |R^N(\varphi f)|^2 \om dV\le c \sum_{k=0}^N \int_{\Bd}|(R^k\varphi)R^{N-k}f|^2 \om dV. \end{align}

 For standard weights $\om$ and for so-called Bekoll\'{e}-Bonami weights it has been shown in \cite{OrtFab} and \cite{CasFab_Bekolleweights} that the right hand side of this is bounded by $c\|f\|^2_{B^N_\om}$, if and only if the terms of the sum corresponding to $k=0$  and $k=N$ are bounded by $c\|f\|^2_{B^N_\om}$, and that these two conditions together characterize $\Mult(B^N_\om)$. Note that these conditions can be equivalently expressed as $\varphi\in H^\infty$ and  $|R^N\varphi|^2 \om dV$ is a $B^N_\om$-Carleson measure. We show in \cite{AlHaMcCRiWeightedB} that the same is true for all admissible radial weights. In fact, it is a rather short argument that shows that
  a bound on the left hand side of (\ref{LeibnitzInequality}) implies a bound on the right hand side of (\ref{LeibnitzInequality}), and this argument is valid  for all  weighted Besov spaces $B^N_\om$ that satisfy  a  scalar version of the multiplier inclusion condition (see Definition  \ref{IntroMultInclCond}).  It turns out that the vector-valued versions of these results are true as well, and that will be an important ingredient in the proof of Theorem \ref{RowColumnOp}. We start by setting up the notation that we will use. A part of this involves extending the definitions given in the Introduction to multipliers between different spaces.

If $\HE$ is a separable Hilbert space and if $\HH$ is a reproducing kernel Hilbert space on $\Bd$, then we will identify $\HH\otimes \HE$ with a space $\HH(\HE)$ of $\HE$-valued functions on $\Bd$, where the identification is given by $f(z)x\cong f(z) \otimes x$ for $f\in \HH$ and $x\in \HE$. We will use the notations $\HH\otimes \HE$ and $\HH(\HE)$ interchangeably.

Let  $\HH$ and $\HK$ be Hilbert spaces of analytic functions.  We will write
 $$\Mult(\HH,\HK)=\{\varphi: \varphi \HH \subseteq  \HK\}.$$
Then any sequence $\Phi=\{\varphi_1, \varphi_2,...\}\subseteq  \Mult(\HH, \HK)$ of multipliers can be used to define a column operator  $\Phi^C: h\to (\varphi_1h, \varphi_2 h,...)^T$ and a row operator $\Phi^R:(h_1,h_2,...)^T \to \sum_{i\ge 1}\varphi_ih_i$. Here we have used $(h_1,...)^T$ to denote a transpose of a row vector. We write $M^C(\HH,\HK)$  for the set of bounded column multiplication operators $\HH \to \bigoplus_{n=1}^\infty \HK$ and $M^R(\HH,\HK)$ for the set of bounded row multiplication operators $\bigoplus_{n=1}^\infty \HH \to \HK$. Thus $\Phi^C\in M^C(\HH,\HK)$ if and only if there is a $c>0$ such that $$\sum_{j=1}^\infty \|\varphi_j h\|^2_{\HK} \le c\|h\|^2_{\HH} \text{ for all }h \in \HH,$$ and $\Phi^R\in M^R(\HH,\HK)$ if and only if there is a $c>0$ such that $$\|\sum_{j=1}^\infty \varphi_j h_j\|^2_{\HK} \le c\sum_{j=1}^\infty \|h_j\|^2_{\HH} \text{ for all }h_j \in \HH.$$ We will write $\|\Phi^C\|_{(\HH,\HK)}$ and $\|\Phi^R\|_{(\HH,\HK)}$ for the norms of these operators. Note that by considering the components of $\Phi \in \Mult(\HH,\HK(\ell_2))$ with respect to the standard orthonormal basis of $\ell_2$, we obtain an  identification between $M^C(\HH,\HK)$ and $\Mult(\HH,\HK(\ell_2))$. In the remainder of this paper we will frequently pass back and forth between these different viewpoints. We just need to remember that when we are given $\Phi\in \Mult(\HH,\HK(\ell_2))$ and we want to consider a row operator induced by $\Phi$ that we have to fix a particular orthonormal basis.

We will now set up the  framework for the spaces for which our results hold. Recall from the Introduction that
a Hilbert  space $\HH$ of functions on $\Bd$ will be called a {\it weighted Besov space}, if there is an admissible weight $\om$ on $\Bd$ and a nonnegative integer $N$ such that $\HH=B^N_\om$ with equivalence of norms. Note that it is possible for a weighted Besov space to have $\HH=B^N_\om=B^K_{\tilde{\om}}$ for $N\ne K$ and $\om \ne \tilde{\om}$. In fact, in \cite{AlHaMcCRiWeightedB} we show that for each admissible radial weight $\om$ there is a one parameter family $\{\om_s\}_{s\ge 0}$ of admissible weights such that  $B^N_\om= B^{N-s}_{\om_s}$ for all $s\ge 0$.

 Since the weight function $\om$ is integrable it follows that $H^\infty(\Bd)\subseteq  B^0_\om$, and hence any  weighted Besov space contains the polynomials. This implies in particular that $k_z(z) \ne 0$ for all $z\in \Bd$, whenever $k$ is the reproducing kernel of any weighted Besov space.

Let $\HH$ and $\HK$ be weighted Besov spaces. We say that the pair $(\HH,\HK)$ satisfies the {\it multiplier inclusion condition}, if there are admissible weights
 $\om$ and $\tilde{\om}$  and $N\in \N$ such that $\HH=B^N_\om$, $\HK=B^N_{\tilde{\om}}$ and
$$\Mult(B^N_{\om},B^N_{\tilde{\om}}(\ell_2))\subseteq  \Mult(B^{N-1}_{\om},B^{N-1}_{\tilde{\om}}(\ell_2))\subseteq  \cdots \subseteq  \Mult(B^0_{\om},B^0_{\tilde{\om}}(\ell_2)),$$ and if the inclusions are continuous, i.e. whenever $ 1\le n \le N,$ then  there is a $c>0$ such that
 for all $ \Phi\in \Mult(B^n_{\om},B^n_{\tilde{\om}}(\ell_2))$ we have
 \begin{align}\label{3.3}  \|\Phi\|_{\Mult(B^{n-1}_{\om},B^{n-1}_{\tilde{\om}}(\ell_2))} \le c\|\Phi\|_{\Mult(B^n_{\om},B^n_{\tilde{\om}}(\ell_2))}.
\end{align}

We will say that the pair $(\HH,\HK)$ satisfies the scalar multiplier inclusion condition, if there are admissible weights
 $\om$ and $\tilde{\om}$  and $N\in \N$ such that $\HH=B^N_\om$, $\HK=B^N_{\tilde{\om}}$ and
$$\Mult(B^N_{\om},B^N_{\tilde{\om}})\subseteq  \Mult(B^{N-1}_{\om},B^{N-1}_{\tilde{\om}})\subseteq  \cdots \subseteq  \Mult(B^0_{\om},B^0_{\tilde{\om}})$$ with continuous inclusions.

We mentioned in the Introduction that a well-known approach to verify that a weighted Besov space satisfies the multiplier inclusion condition uses the functorial property of the complex interpolation method. At this point we will also indicate a somewhat more elementary method that can be used to verify that a pair $(\HH,\HK)$ satisfies the multiplier inclusion condition (\ref{3.3}). This method works sometimes, when simple formulas for the reproducing kernels of the spaces are known. We will prove that $(B^s_{\mathbf{1}}, B^t_{\mathbf{1}})$ satisfies
the multiplier inclusion condition, whenever $t\le s <(d+1)/2$. This includes the interesting case where $t=s=d/2$ and the spaces both equal the Drury-Arveson space $H^2_d$.

For real values of $s<(d+1)/2$ the function $k^s_w(z)=\frac{1}{(1-\la z,w\ra)^{d+1-2s}}$ defines a reproducing kernel on $\Bd$. Let $\HH_s$ be the unique Hilbert function space with reproducing kernel $k^s$. So, for example $\HH_{d/2}=H^2_d$, $H_{1/2}=H^2(\dB)$, and $\HH_0=L^2_a(1)$. Then a routine calculation of the norms of functions in $\HH_s$ shows that for $f\in \HH_s$ one has
\begin{equation}\label{derivative} \|f\|_{\HH_s}^2 \approx |f(0)|^2 + \|Rf\|^2_{\HH_{s-1}}, \end{equation} and also $\HH_s=B^s_{\mathbf{1}}$ with equivalence of norms, see e.g. \cite{RiSunkes}, Section 2. Furthermore, if $s<1/2$, then Theorem 2.7 of \cite{ZhuBall} states that
%$\HH_s=L^2_a(c_s (1-|z|^2)^{-2s})$, $c_s=\frac{\Gamma(d+1-2s)}{d!\Gamma(1-2s)} $ with equality of norms.
$\HH_s=L^2_a((1-|z|^2)^{-2s})$ with equivalence of norms.

Now fix $t\le s <(d+1)/2$. Choose  a positive integer $N >s-1/2$ and set $\om(z)= (1-|z|^2)^{2(N-s)}$ and $\tilde{\om}(z)= (1-|z|^2)^{2(N-t)}$. Then $\om, \tilde{\om}$ are admissible weights and since $s-N<1/2$ we have $B^{0}_\om=\HH_{s-N}$. Then by (\ref{derivative}) we conclude that for integers $k$ with $0 \le k \le N$ we have $B^k_\om= \HH_{s-N+k}$ with equivalence of norms. Similarly, $B^k_{\tilde{\om}}= \HH_{t-N+k}$, whenever  $0 \le k \le N$. Then we must show the continuous inclusions
$$\Mult(\HH_{s-k},\HH_{t-k}(\ell_2)) \subseteq \Mult(\HH_{s-k-1},\HH_{t-k-1}(\ell_2)), \ k= 0, 1, ..N-1.$$

 We establish the scalar version $$\Mult(\HH_{s-k},\HH_{t-k}) \subseteq \Mult(\HH_{s-k-1},\HH_{t-k-1}), \ k= 0, 1, ..N-1. $$ By substituting $s$ for $s-k$ and $t$ for $t-k$ we see that it is sufficient to prove
 $$\Mult(\HH_{s},\HH_{t}) \subseteq \Mult(\HH_{s-1},\HH_{t-1}) \ \text{ whenever } t \le s<(d+1)/2. $$

  Let $\varphi \in \Mult(\HH_s,\HH_t)$ of multiplier norm $\le 1$. Then $k^t_w(z)-\varphi(z)\overline{\varphi(w)}k^s_w(z)$ is positive definite. Note that $(1-\la z,w\ra)^{-2}$ is positive definite and $(1-\la z,w\ra)^{-2}k^s_w(z)=k^{s-1}_w(z)$ and $(1-\la z,w\ra)^{-2}k^t_w(z)=k^{t-1}_w(z)$. Thus by the Schur product theorem we can conclude that $ k^{t-1}_w(z)-\varphi(z)\overline{\varphi(w)}k^{s-1}_w(z)$ is positive definite. That is equivalent to saying that $\varphi$ is a contractive multiplier from $\HH_{s-1}$ into $\HH_{t-1}$.

  Thus the scalar  multiplier inclusion condition (\ref{3.3}) holds, and the column vector version
 can be shown similarly.

 The following lemma says that the multiplier inclusion condition (\ref{3.3}) implies that given a bound on the left hand side of   inequality (\ref{LeibnitzInequality}) one also has a bound on the right hand side of that inequality.
 \begin{lemma} \label{Leibnitz estimate}  Let $\om$ and $\tilde{\om}$ be admissible weights, let $N\in \N$, and suppose that there is $c>0$ such that $(B^N_\om,B^N_{\tilde{\om}})$ satisfies the multiplier inclusion condition
 \begin{align*}  \|\Phi\|_{\Mult(B^{n-1}_{\om},B^{n-1}_{\tilde{\om}}(\ell_2))} \le c\|\Phi\|_{\Mult(B^n_{\om},B^n_{\tilde{\om}}(\ell_2))}\end{align*} for all $ \Phi\in \Mult(B^n_{\om},B^n_{\tilde{\om}}(\ell_2))$ and $ 1\le n\le N$.

 Then there is a $C>0$ such that whenever $\Phi= \{\varphi_i\}_{i\in \N}\in M^C(B^N_\om,B^N_{\tilde{\om}})$ satisfies $\|\Phi^C\|_{(B^N_\om,B^N_{\tilde{\om}})}\le 1$,
then for all  integers $j,k\ge 0$ with $j+k\le N$ and each $h\in \mathcal B^N_\om$  we have
$$\sum_{i=1}^\infty \|(R^j\varphi_i)R^{k}h\|^2_{B^{N-(j+k)}_{\tilde{\om}}}\le C \|h\|^2_{B^N_\om}.$$
\end{lemma}

Of particular interest is the case $k=0$. In compact form it implies that under the hypothesis of the Lemma there is $c>0$ such that for each $j$ with $0\le j\le N$
$$\|R^j\Phi\|_{\Mult(B^N_\om,B^{N-j}_{\tilde{\om}}(\ell_2))}\le c  \|\Phi\|_{\Mult(B^N_\om,B^N_{\tilde{\om}}(\ell_2))}.$$

\begin{proof} Suppose $\|\Phi\|_{\Mult(B^N_\om,B^{N}_{\tilde{\om}}(\ell_2))}= \|\Phi^C\|_{(B^N_\om,B^{N}_{\tilde{\om}})}\le 1$. The multiplier inclusion condition  implies that there is $c>0$ such that for each $ 0 \le k\le n$ we have $\|\Phi\|_{\Mult(B^{N-k}_\om,B^{N-k}_{\tilde{\om}}(\ell_2))} \le c$ and hence for each $h\in B^N_\om$ we have
 \begin{align*} \sum_{i=1}^\infty \|\varphi_iR^{k}h\|^2_{B^{N-k}_{\tilde{\om}}} \le c \|R^{k} h\|^2_{B^{N-k}_\om} \lesssim c\|h\|^2_{B^N_\om}.
 \end{align*}

Thus the Lemma holds for $j=0$ and any $0 \le k\le N$, and we claim that the case of $j>0$ can be reduced to the case of $j=0$.

If $j>0$ and $j+k \le N$, then $$ (R^j\varphi_i) R^kh= R((R^{j-1}\varphi_i)R^kh)-(R^{j-1}\varphi_i) R^{k+1}h$$ and hence
\begin{align*} \sum_{i=1}^\infty &\|(R^j\varphi_i)R^{k}h\|^2_{B_{\tilde{\om}}^{N-(j+k)}} \\ &\le 2\sum_{i=1}^\infty \|R((R^{j-1}\varphi_i)R^{k}h)\|^2_{B_{\tilde{\om}}^{N-(j+k)}}+ \|(R^{j-1}\varphi_i) R^{k+1}h\|^2_{B_{\tilde{\om}}^{N-(j+k)}}\\
&\lesssim \sum_{i=1}^\infty \|(R^{j-1}\varphi_i)R^{k}h\|^2_{B_{\tilde{\om}}^{N-(j-1+k)}}+ \|(R^{j-1}\varphi_i) R^{k+1}h\|^2_{B_{\tilde{\om}}^{N-((j-1)+k+1)}}.
 \end{align*}
This means that the proof of the Lemma for $j>0$ has been reduced to the case of $j-1$. Hence finitely many iterations of this argument conclude the proof.
\end{proof}

  \begin{theorem}\label{RowColumnOp}  Let    $\HH$ and $\HK$ be weighted  Besov spaces such that $(\HH,\HK)$ satisfies the multiplier inclusion condition (\ref{3.3}).

   Then  there is a $c>0$ such that $$\|\Phi^R\|_{(\HH,\HK)}\le c\|\Phi^C\|_{(\HH,\HK)}$$ for all $\Phi=\{\varphi_1, \varphi_2,...\}\in   M^C(\HH,\HK)$.
  \end{theorem}

\begin{proof}  By hypothesis we may choose $N\in \N$ and admissible weights $\om, \tilde{\om}$ so that $\HH=B^N_\om$, $\HK=B^N_{\tilde{\om}}$, and so that Lemma \ref{Leibnitz estimate} applies.

 Let $\Phi=\{\varphi_i\}$ be a sequence of analytic functions on $\Bd$ with  $$\|\Phi^C h\|^2_{\HK(\ell_2)}=\sum_{i=1}^\infty \|\varphi_i h\|^2_{\HK} \le \|h\|^2_{\HH} \text{ for all }h \in \HH. $$ We have to show the existence of $c>0$ such that
$$\|\sum_{j=1}^\infty \varphi_j h_j\|^2_{\HK} \le c \sum_{j=1}^\infty \|h_j\|^2_{\HH} \text{ for all } \{h_j\}\in \bigoplus_{j=1}^\infty \HH.$$
We have $$\|f\|^2_\HK \approx |f(0)|^2+ \|R^Nf\|^2_{L^2_a(\tilde{\om})}.$$

Let $k^1_z$ be the reproducing kernel for $\HH$ and $k^2_z$ be the reproducing kernel for $\HK$. Then
$$\sum_{i=1}|\varphi_i(z)|^2 = \sum_{i=1}\frac{|\la\varphi_ik^1_z, k^2_z\ra_\HK|^2}{\|k^1_z\|^4_{\HH}}\le \frac{ \|k^2_z\|^2_{\HK}}{\|k^1_z\|^2_{\HH}}.$$
This implies that for any   $ \{h_j\}\in \bigoplus_{j=1}^\infty \HH$
$$\sum_{j=1}^\infty |\varphi_j(z) h_j(z)|\le \frac{ \|k^2_z\|_{\HK}}{\|k^1_z\|_{\HH}}\left(\sum_{j=1}^\infty|h_j(z)|^2\right)^{1/2}\le \|k^2_z\|_{\HK} \left( \sum_{j=1}^\infty\|h_j\|_{\HH}^2\right)^{1/2}.$$

Taking $z=0$ we note that it suffices to show that $$\|R^N(\sum_{j=1}^\infty \varphi_j h_j)\|_{L^2_a(\tilde{\om})}^2 \le C\sum_{j=1}^\infty \|h_j\|_{\HH}^2.$$

 By the above the series $\sum_{j=1}^\infty \varphi_j h_j$ converges uniformly on compact subsets of $\Bd$, thus by analyticity and the Leibnitz rule we have
\begin{align*}|R^N \left(\sum_{j=1}^\infty \varphi_j h_j\right)|^2
&= |\sum_{k=0}^N \left(\begin{matrix}N\\k\end{matrix}\right)\sum_{j=1}^\infty (R^k\varphi_j) (R^{N-k}h_j)|^2\\
&\le c \sum_{k=0}^N \left( \sum_{j=1}^\infty |R^k\varphi_j| |R^{N-k}h_j|\right)^2\\
&=c \sum_{k=0}^N  \sum_{j=1,i=1}^\infty |R^k\varphi_j| |R^{N-k}h_j| |R^k\varphi_i| |R^{N-k}h_i|\\
&\le c \sum_{k=0}^N \sum_{j=1,i=1}^\infty |R^k\varphi_j R^{N-k}h_i|^2.
\end{align*}
Now we integrate both sides of the inequality over $\Bd$ against the measure $\tilde{\om} dV$ and obtain
\begin{align*}\|R^N \left(\sum_{j=1}^\infty \varphi_j h_j\right)\|_{L^2_a(\tilde{\om})}^2&\le c \sum_{k=0}^N \sum_{i=1}^\infty \left(\sum_{j=1}^\infty\|R^k\varphi_jR^{N-k}h_i\|_{L^2_a(\tilde{\om})}^2\right) \\
&\lesssim \sum_{k=0}^N \sum_{i=1}^\infty \|h_i\|^2_{B^N_\om}\lesssim \sum_{i=1}^\infty \|h_i\|^2_\HH \end{align*} by Lemma \ref{Leibnitz estimate}.
\end{proof}

\subsection{A bounded row, but unbounded column operator}
In this Section we will show that if $d\ge 2$ then there are sequences of multipliers $\Phi=\{\varphi_1,\varphi_2,...\}$ of the Drury-Arveson space $H^2_d$ such that $\Phi^R$ is bounded $\bigoplus_{j=1}^\infty H^2_d\to H^2_d$, but $\Phi^C$ is unbounded  $H^2_d\to \bigoplus_{j=1}^\infty H^2_d$.
The proof is patterned after an example of Trent for the Dirichlet space, \cite{TrentCorona}. The case at hand requires more work due to the fact that for the Drury-Arveson kernel $k$ the expression $1-1/k_w(z)$ is a sum of only finitely many terms of the form $\varphi_i(z) \overline{\varphi_i(w)}$.

For $n \in \N$ let $S_n=\{\mu=(\mu_1,...,\mu_n): \mu_j\in \{1,2,...,d\}\}$ be the set of $n-$tuples of elements in $\{1,...,d\}$. Set $S=\bigcup_{n=1}^\infty S_n$ and for $\mu \in S$ write $l(\mu)=n$ if $\mu\in S_n$, the length of $\mu$.

Since $S$ is countable it will suffice to exhibit a family of functions $\{\varphi_\mu\}_{\mu\in S}$ on $\Bd$ such that
$$\|\sum_{\mu\in S} \varphi_\mu h_\mu\|^2 \le C \sum_{\mu\in S}\|h_\mu\|^2 \text{ whenever } h_\mu \in H^2_d \text{ and }
\sum_{\mu\in S} \|\varphi_\mu\|^2 =\infty.$$

For $\mu\in S_n$   and $z\in \C^d$ set $z_\mu= z_{\mu_1}z_{\mu_2}\cdots z_{\mu_n}$.

Define
$\varphi_\mu(z)=\frac{1}{l(\mu)}z_{\mu}$. Note that for each $\mu\in S_n$ there is a multi index $\alpha=(\alpha_1,...,\alpha_d)\in \N_0^d$ with $|\alpha|=n$ and such that $\varphi_\mu(z)=\frac{1}{n} z^\alpha$. Conversely, if $\alpha\in \N_0^d$ is a multi index with $|\alpha|=n$, then there are $\frac{|\alpha|!}{\alpha!}$ distinct $\mu \in S_n$ with $\varphi_\mu(z)=\frac{1}{n} z^\alpha$. Recall that in $H^2_d$ we have $\|z^\alpha\|^2= \frac{\alpha !}{|\alpha|!}$.

Thus
\begin{align*}\sum_{\mu\in S} \|\varphi_\mu\|^2
&=\sum_{n=1}^\infty \frac{1}{n^2} \sum_{\mu\in S_n}\|z_{\mu_1}z_{\mu_2}\cdots z_{\mu_n}\|^2\\
&=\sum_{n=1}^\infty \frac{1}{n^2} \sum_{|\alpha|=n}\frac{|\alpha| !}{\alpha!} \|z^\alpha\|^2 \\
&=\sum_{n=1}^\infty \frac{1}{n^2} \sum_{|\alpha|=n} 1\\
&\ge \sum_{n=1}^\infty \frac{1}{n} =\infty
\end{align*}
since for each $n \in \N$ there are $n$ multi-indices $\alpha$ of the type  $\alpha=(k,n-k,0,...,0)$. Here we use $d\ge 2$.

On the other hand consider
\begin{align*} \|\sum_{\mu\in S}^\infty \varphi_\mu h_\mu\|^2 = \|\sum_{n=1}^\infty \frac{1}{n}\sum_{\mu\in S_n} z_\mu h_\mu\|^2 \le  \left(\sum_{n=1}^\infty \frac{1}{n^2}\right)\left(\sum_{n=1}^\infty \|\sum_{\mu\in S_n}z_\mu h_\mu\|^2\right)
\end{align*}

We now show by induction that for each $n\ge 1$ we have $$\|\sum_{\mu\in S_n}z_\mu h_\mu\|^2 \le \sum_{\mu\in S_n}\|h_\mu\|^2$$ and that will finish the proof.

 The case $n=1$ follows since
 $M_z$ is a $d$-contraction, i.e. $\|\sum_{k=1}^d z_k f_k\|^2 \le \sum_{k=1}^d \|f_k\|^2$. This shows that $\|\sum_{\mu\in S_1}z_\mu h_\mu\|^2\le \sum_{\mu\in S_1} \|h_\mu\|^2$. Now suppose that the claim holds for some $n\ge 1$. We will show that it holds for $n+1$. Note that $S_{n+1}=S_1 \times S_n$, hence  as above

 \begin{align*}
 \|\sum_{\mu\in S_{n+1}}z_\mu h_\mu\|^2&= \|\sum_{k=1}^dz_k \sum_{\mu'\in S_n} z_{\mu'}h_{(k,\mu')}\|^2\\
 &\le \sum_{k=1}^d\| \sum_{\mu'\in S_n} z_{\mu'}h_{(k,\mu')}\|^2\\
 &\le \sum_{k=1}^d \sum_{\mu'\in S_n}\|h_{(k,\mu')}\|^2= \sum_{\mu\in S_{n+1}}\|h_\mu\|^2
 \end{align*} by the induction hypothesis.

Note added in Proof: The recent paper \cite{JuryMartin} shows that if a complete Pick space $\HH$
satisfies $M^C(\HH)\subseteq M^R(\HH)$, then $\HH \odot \HH = \HH \cdot \HH =\{fg: f, g\in \HH\}$.

\include{biblio}
\bibliography{Biblio_WeakProduct_Quotients}

\begin{thebibliography}{10}

\bibitem{AgMcC}
Jim Agler and John~E. McCarthy.
\newblock {\em Pick interpolation and {H}ilbert function spaces}, volume~44 of
  {\em Graduate Studies in Mathematics}.
\newblock American Mathematical Society, Providence, RI, 2002.

\bibitem{AlHaMcCRiInter}
Alexandru Aleman, Michael Hartz, John McCarthy, and Stefan Richter.
\newblock Interpolating sequences in spaces with the complete {P}ick property.
\newblock {\em International Mathematics Research Notices, to appear,
  arXiv:1701.04885}.

\bibitem{AlHaMcCRiWeightedB}
Alexandru Aleman, Michael Hartz, John McCarthy, and Stefan Richter.
\newblock Radially weighted {B}esov spaces and the {P}ick property.
\newblock {\em arXiv:1807.00730}.

\bibitem{AlHaMcCRiM}
Alexandru Aleman, Michael Hartz, John~E. McCarthy, and Stefan Richter.
\newblock The {S}mirnov class for spaces with the complete {P}ick property.
\newblock {\em J. Lond. Math. Soc. (2)}, 96(1):228--242, 2017.

\bibitem{AlHMcCr_Factors}
Alexandru Aleman, Michael Hartz, John~E. McCarthy, and Stefan Richter.
\newblock Factorizations induced by complete {N}evanlinna--{P}ick factors.
\newblock {\em Adv. Math.}, 335:372--404, 2018.

\bibitem{AlRiSunkesCyclicVector}
Alexandru Aleman, Stefan Richter, and James Sunkes.
\newblock Applications of the {P}ick-{S}mirnov class to questions about cyclic
  functions.
\newblock {\em In preparation}.

\bibitem{ARSWBilinearForms}
Nicola Arcozzi, Richard Rochberg, Eric Sawyer, and Brett~D. Wick.
\newblock Bilinear forms on the {D}irichlet space.
\newblock {\em Anal. PDE}, 3(1):21--47, 2010.

\bibitem{BerghLoefstroem}
J\"oran Bergh and J\"orgen L\"ofstr\"om.
\newblock {\em Interpolation spaces. {A}n introduction}.
\newblock Springer-Verlag, Berlin-New York, 1976.
\newblock Grundlehren der Mathematischen Wissenschaften, No. 223.

\bibitem{BrownShields}
Leon Brown and Allen~L. Shields.
\newblock Cyclic vectors in the {D}irichlet space.
\newblock {\em Trans. Amer. Math. Soc.}, 285(1):269--303, 1984.

\bibitem{CasFab_Bekolleweights}
Carme Cascante and Joan F\`abrega.
\newblock Bilinear forms on weighted {B}esov spaces.
\newblock {\em J. Math. Soc. Japan}, 68(1):383--403, 2016.

\bibitem{CasFabOrt}
Carme Cascante, Joan F{\`a}brega, and Joaqu{\'{\i}}n~M. Ortega.
\newblock On weighted {T}oeplitz, big {H}ankel operators and {C}arleson
  measures.
\newblock {\em Integral Equations Operator Theory}, 66(4):495--528, 2010.

\bibitem{CRW}
R.~R. Coifman, R.~Rochberg, and Guido Weiss.
\newblock Factorization theorems for {H}ardy spaces in several variables.
\newblock {\em Ann. of Math. (2)}, 103(3):611--635, 1976.

\bibitem{DavidsonHamilton2011}
Kenneth~R. Davidson and Ryan Hamilton.
\newblock Nevanlinna-{P}ick interpolation and factorization of linear
  functionals.
\newblock {\em Integral Equations Operator Theory}, 70(1):125--149, 2011.

\bibitem{DavidsonRamseyShalit2015}
Kenneth~R. Davidson, Christopher Ramsey, and Orr~Moshe Shalit.
\newblock Operator algebras for analytic varieties.
\newblock {\em Trans. Amer. Math. Soc.}, 367(2):1121--1150, 2015.

\bibitem{EffrosRuan}
Edward~G. Effros and Zhong-Jin Ruan.
\newblock {\em Operator spaces}, volume~23 of {\em London Mathematical Society
  Monographs. New Series}.
\newblock The Clarendon Press, Oxford University Press, New York, 2000.

\bibitem{DiriPrimer}
Omar El-Fallah, Karim Kellay, Javad Mashreghi, and Thomas Ransford.
\newblock {\em A primer on the {D}irichlet space}, volume 203 of {\em Cambridge
  Tracts in Mathematics}.
\newblock Cambridge University Press, Cambridge, 2014.

\bibitem{JuryMartin}
Michael~T. Jury and Robert~T.W. Martin.
\newblock Factorization in weak products of complete {P}ick spaces.
\newblock {\em arXiv:1806.05268}.

\bibitem{Luo}
Shuaibing Luo.
\newblock On the index of invariant subspaces in the space of weak products of
  {D}irichlet functions.
\newblock {\em Complex Anal. Oper. Theory}, 9(6):1311--1323, 2015.

\bibitem{LuoRi}
Shuaibing Luo and Stefan Richter.
\newblock Hankel operators and invariant subspaces of the {D}irichlet space.
\newblock {\em J. Lond. Math. Soc. (2)}, 91(2):423--438, 2015.

\bibitem{MarcusSpielmanSrivastava}
Adam~W. Marcus, Daniel~A. Spielman, and Nikhil Srivastava.
\newblock Interlacing families {II}: {M}ixed characteristic polynomials and the
  {K}adison-{S}inger problem.
\newblock {\em Ann. of Math. (2)}, 182(1):327--350, 2015.

\bibitem{OrtFab}
Joaqu{\'{\i}}n~M. Ortega and Joan F{\`a}brega.
\newblock Pointwise multipliers and decomposition theorems in analytic {B}esov
  spaces.
\newblock {\em Math. Z.}, 235(1):53--81, 2000.

\bibitem{RichterDiss}
Stefan Richter.
\newblock Invariant subspaces in {B}anach spaces of analytic functions.
\newblock {\em Trans. Amer. Math. Soc.}, 304(2):585--616, 1987.

\bibitem{RiSuWeakProd}
Stefan Richter and Carl Sundberg.
\newblock Weak products of {D}irichlet functions.
\newblock {\em J. Funct. Anal.}, 266(8):5270--5299, 2014.

\bibitem{RiSunkes}
Stefan Richter and James Sunkes.
\newblock Hankel operators, invariant subspaces, and cyclic vectors in the
  {D}rury-{A}rveson space.
\newblock {\em Proc. Amer. Math. Soc.}, 144(6):2575--2586, 2016.

\bibitem{RiWi}
Stefan Richter and Brett~D. Wick.
\newblock A remark on the multipliers on spaces of weak products of functions.
\newblock {\em Concr. Oper.}, 3:25--28, 2016.

\bibitem{RyanBanachTensor}
Raymond~A. Ryan.
\newblock {\em Introduction to tensor products of {B}anach spaces}.
\newblock Springer Monographs in Mathematics. Springer-Verlag London, Ltd.,
  London, 2002.

\bibitem{Shalit}
Orr Shalit.
\newblock Operator theory and function theory in the {D}rury-{A}rveson space
  and its quotients.
\newblock In {\em Operator Theory (2015)}, pages 1125--1180. Springer, Basel.

\bibitem{TrentCorona}
Tavan~T. Trent.
\newblock A corona theorem for multipliers on {D}irichlet space.
\newblock {\em Integral Equations Operator Theory}, 49(1):123--139, 2004.

\bibitem{ZhaoZhu}
Ruhan Zhao and Kehe Zhu.
\newblock Theory of {B}ergman spaces in the unit ball of {$\Bbb C^n$}.
\newblock {\em M\'em. Soc. Math. Fr. (N.S.)}, (115):vi+103 pp. (2009), 2008.

\bibitem{ZhuBall}
Kehe Zhu.
\newblock {\em Spaces of holomorphic functions in the unit ball}, volume 226 of
  {\em Graduate Texts in Mathematics}.
\newblock Springer-Verlag, New York, 2005.

\end{thebibliography}

 \end{document}